\documentclass[11pt]{amsart}
\usepackage{amsmath,amsthm,amssymb,latexsym,epic,bbm,comment}
\usepackage{graphicx,enumerate,stmaryrd, xcolor, tikz-cd}
\usepackage{ytableau}
\usepackage[all,2cell]{xy}
\xyoption{2cell}
\usepackage[active]{srcltx}
\usepackage[parfill]{parskip}
\usepackage[
colorlinks=true,
linkcolor=black, 
anchorcolor=black,
citecolor=black, 
urlcolor=black, 
]{hyperref}
\usepackage[e]{esvect}
\usepackage{mathrsfs}
\usepackage{todonotes}
\usepackage{enumitem}
\usepackage{booktabs}
\usepackage{makecell}
\usepackage{array}

\newcolumntype{L}[1]{>{\raggedright\arraybackslash}p{#1}}

\newtheorem{theorem}{Theorem}[section]

\newtheorem{lemma}[theorem]{Lemma}

\theoremstyle{definition}
\newtheorem{definition}[theorem]{Definition}
\newtheorem{remark}[theorem]{Remark}
\newtheorem{example}[theorem]{Example}
\newtheorem{conv}[theorem]{Convention}

\usepackage{tikz}
\usetikzlibrary{arrows.meta}

\tikzset{
dynkin node/.style={
circle,
draw,
fill=white,
inner sep=0pt,
minimum size=4.2pt
},
dynkin label/.style={
draw=none,
fill=none,
inner sep=0pt,
font=\scriptsize
},
dynkin edge/.style={
line width=0.45pt
},
dynkin double edge/.style={
line width=0.45pt,
double distance=1.1pt
},
dynkin double arrow/.style={
dynkin double edge,
-{Stealth[length=4pt,width=10pt]}
}
}

\newcommand{\dynA}{%
\begin{tikzpicture}[baseline=-0.5ex, x=0.55cm, y=0.55cm]
\node[dynkin node] (s1) at (0,0) {};
\node[dynkin node] (s2) at (1,0) {};
\node[dynkin node] (s3) at (2,0) {};
\node[dynkin label] at (3,0) {$\cdots$};
\node[dynkin node] (snm1) at (4,0) {};
\node[dynkin node] (sn) at (5,0) {};

\draw[dynkin edge] (s1) -- (s2);
\draw[dynkin edge] (s2) -- (s3);
\draw[dynkin edge] (s3) -- (2.65,0);
\draw[dynkin edge] (3.35,0) -- (snm1);
\draw[dynkin edge] (snm1) -- (sn);

\node[dynkin label] at (0,-0.42) {$1$};
\node[dynkin label] at (1,-0.42) {$2$};
\node[dynkin label] at (2,-0.42) {$3$};
\node[dynkin label] at (4,-0.42) {$n-1$};
\node[dynkin label] at (5,-0.42) {$n$};
\end{tikzpicture}%
}

\newcommand{\dynB}{%
\begin{tikzpicture}[baseline=-0.5ex, x=0.55cm, y=0.55cm]
\node[dynkin node] (s1) at (0,0) {};
\node[dynkin node] (s2) at (1,0) {};
\node[dynkin node] (s3) at (2,0) {};
\node[dynkin label] at (3,0) {$\cdots$};
\node[dynkin node] (snm1) at (4,0) {};
\node[dynkin node] (sn) at (5,0) {};

\draw[dynkin double arrow] (s1) -- (s2);
\draw[dynkin edge] (s2) -- (s3);
\draw[dynkin edge] (s3) -- (2.65,0);
\draw[dynkin edge] (3.35,0) -- (snm1);
\draw[dynkin edge] (snm1) -- (sn);

\node[dynkin label] at (0,-0.42) {$1$};
\node[dynkin label] at (1,-0.42) {$2$};
\node[dynkin label] at (2,-0.42) {$3$};
\node[dynkin label] at (4,-0.42) {$n-1$};
\node[dynkin label] at (5,-0.42) {$n$};
\end{tikzpicture}%
}

\newcommand{\dynC}{%
\begin{tikzpicture}[baseline=-0.5ex, x=0.55cm, y=0.55cm]
\node[dynkin node] (s1) at (0,0) {};
\node[dynkin node] (s2) at (1,0) {};
\node[dynkin node] (s3) at (2,0) {};
\node[dynkin label] at (3,0) {$\cdots$};
\node[dynkin node] (snm1) at (4,0) {};
\node[dynkin node] (sn) at (5,0) {};

\draw[dynkin double arrow] (s2) -- (s1);
\draw[dynkin edge] (s2) -- (s3);
\draw[dynkin edge] (s3) -- (2.65,0);
\draw[dynkin edge] (3.35,0) -- (snm1);
\draw[dynkin edge] (snm1) -- (sn);

\node[dynkin label] at (0,-0.42) {$1$};
\node[dynkin label] at (1,-0.42) {$2$};
\node[dynkin label] at (2,-0.42) {$3$};
\node[dynkin label] at (4,-0.42) {$n-1$};
\node[dynkin label] at (5,-0.42) {$n$};
\end{tikzpicture}%
}

\newcommand{\dynD}{%
\begin{tikzpicture}[baseline=-0.5ex, x=0.55cm, y=0.55cm]
\node[dynkin node] (s1) at (1,1) {};
\node[dynkin node] (s2) at (0,0) {};
\node[dynkin node] (s3) at (1,0) {};
\node[dynkin node] (s4) at (2,0) {};
\node[dynkin label] at (3,0) {$\cdots$};
\node[dynkin node] (snm1) at (4,0) {};
\node[dynkin node] (sn) at (5,0) {};

\draw[dynkin edge] (s1) -- (s3);
\draw[dynkin edge] (s2) -- (s3);
\draw[dynkin edge] (s3) -- (s4);
\draw[dynkin edge] (s4) -- (2.65,0);
\draw[dynkin edge] (3.35,0) -- (snm1);
\draw[dynkin edge] (snm1) -- (sn);

\node[dynkin label] at (1,1.38) {$1$};
\node[dynkin label] at (0,-0.42) {$2$};
\node[dynkin label] at (1,-0.42) {$3$};
\node[dynkin label] at (2,-0.42) {$4$};
\node[dynkin label] at (4,-0.42) {$n-1$};
\node[dynkin label] at (5,-0.42) {$n$};
\end{tikzpicture}%
}

\newcommand{\dynGtwo}{%
\begin{tikzpicture}[baseline=-0.5ex, x=0.55cm, y=0.55cm]
\node[dynkin node] (s1) at (0,0) {};
\node[dynkin node] (s2) at (1,0) {};

\draw[dynkin double arrow] (s2) -- (s1);
\draw[dynkin edge] (s2) -- (s1);

\node[dynkin label] at (0,-0.42) {$1$};
\node[dynkin label] at (1,-0.42) {$2$};
\end{tikzpicture}%
}

\newcommand{\dynFfour}{%
\begin{tikzpicture}[baseline=-0.5ex, x=0.55cm, y=0.55cm]
\node[dynkin node] (s1) at (0,0) {};
\node[dynkin node] (s2) at (1,0) {};
\node[dynkin node] (s3) at (2,0) {};
\node[dynkin node] (s4) at (3,0) {};

\draw[dynkin edge] (s1) -- (s2);
\draw[dynkin double arrow] (s2) -- (s3);
\draw[dynkin edge] (s3) -- (s4);

\node[dynkin label] at (0,-0.42) {$1$};
\node[dynkin label] at (1,-0.42) {$2$};
\node[dynkin label] at (2,-0.42) {$3$};
\node[dynkin label] at (3,-0.42) {$4$};
\end{tikzpicture}%
}

\newcommand{\dynEsix}{%
\begin{tikzpicture}[baseline=-0.5ex, x=0.55cm, y=0.55cm]
\node[dynkin node] (s1) at (0,0) {};
\node[dynkin node] (s3) at (1,0) {};
\node[dynkin node] (s4) at (2,0) {};
\node[dynkin node] (s5) at (3,0) {};
\node[dynkin node] (s6) at (4,0) {};
\node[dynkin node] (s2) at (2,-1) {};

\draw[dynkin edge] (s1) -- (s3);
\draw[dynkin edge] (s3) -- (s4);
\draw[dynkin edge] (s4) -- (s5);
\draw[dynkin edge] (s5) -- (s6);
\draw[dynkin edge] (s4) -- (s2);

\node[dynkin label] at (0,-0.42) {$1$};
\node[dynkin label] at (1,-0.42) {$3$};
\node[dynkin label] at (2,0.42) {$4$};
\node[dynkin label] at (3,-0.42) {$5$};
\node[dynkin label] at (4,-0.42) {$6$};
\node[dynkin label] at (2,-1.42) {$2$};
\end{tikzpicture}%
}

\newcommand{\dynEseven}{%
\begin{tikzpicture}[baseline=-0.5ex, x=0.55cm, y=0.55cm]
\node[dynkin node] (s1) at (0,0) {};
\node[dynkin node] (s3) at (1,0) {};
\node[dynkin node] (s4) at (2,0) {};
\node[dynkin node] (s5) at (3,0) {};
\node[dynkin node] (s6) at (4,0) {};
\node[dynkin node] (s7) at (5,0) {};
\node[dynkin node] (s2) at (2,-1) {};

\draw[dynkin edge] (s1) -- (s3);
\draw[dynkin edge] (s3) -- (s4);
\draw[dynkin edge] (s4) -- (s5);
\draw[dynkin edge] (s5) -- (s6);
\draw[dynkin edge] (s6) -- (s7);
\draw[dynkin edge] (s4) -- (s2);

\node[dynkin label] at (0,-0.42) {$1$};
\node[dynkin label] at (1,-0.42) {$3$};
\node[dynkin label] at (2,0.42) {$4$};
\node[dynkin label] at (3,-0.42) {$5$};
\node[dynkin label] at (4,-0.42) {$6$};
\node[dynkin label] at (5,-0.42) {$7$};
\node[dynkin label] at (2,-1.42) {$2$};
\end{tikzpicture}%
}

\newcommand{\dynEeight}{%
\begin{tikzpicture}[baseline=-0.5ex, x=0.55cm, y=0.55cm]
\node[dynkin node] (s1) at (0,0) {};
\node[dynkin node] (s3) at (1,0) {};
\node[dynkin node] (s4) at (2,0) {};
\node[dynkin node] (s5) at (3,0) {};
\node[dynkin node] (s6) at (4,0) {};
\node[dynkin node] (s7) at (5,0) {};
\node[dynkin node] (s8) at (6,0) {};
\node[dynkin node] (s2) at (2,-1) {};

\draw[dynkin edge] (s1) -- (s3);
\draw[dynkin edge] (s3) -- (s4);
\draw[dynkin edge] (s4) -- (s5);
\draw[dynkin edge] (s5) -- (s6);
\draw[dynkin edge] (s6) -- (s7);
\draw[dynkin edge] (s7) -- (s8);
\draw[dynkin edge] (s4) -- (s2);

\node[dynkin label] at (0,-0.42) {$1$};
\node[dynkin label] at (1,-0.42) {$3$};
\node[dynkin label] at (2,0.42) {$4$};
\node[dynkin label] at (3,-0.42) {$5$};
\node[dynkin label] at (4,-0.42) {$6$};
\node[dynkin label] at (5,-0.42) {$7$};
\node[dynkin label] at (6,-0.42) {$8$};
\node[dynkin label] at (2,-1.42) {$2$};
\end{tikzpicture}%
}

\newcommand\leftidx[3]{%
{\vphantom{#2}}#1#2#3%
}
\newcommand{\pb}{\leftidx{^p}{b}{}}

\newcommand{\pJ}{\leftidx{^p}{J}{}}
\newcommand{\ph}{\leftidx{^p}{h}{}}
\newcommand{\pmk}{\leftidx{^p}{m}{}}
\newcommand{\pmu}{\leftidx{^p}{\mu}{}}

\newcommand{\starop}[3]{\leftidx{^{*_{#1,#2}}}{#3}{}}
\newcommand{\Jtop}{\overline{J_1}}

\newcommand{\Hom}{\mathrm{Hom}}

\newcommand{\rA}{\mathrm{A}}
\newcommand{\rB}{\mathrm{B}}
\newcommand{\rC}{\mathrm{C}}
\newcommand{\rD}{\mathrm{D}}

\newcommand{\rH}{\mathrm{H}}

\newcommand{\J}{\mathcal{J}}
\def\H{{\mathcal{H}}}

\newcommand{\BbbK}{\mathbb{K}}

\newcommand{\bbZ}{\mathbb{Z}}

\numberwithin{equation}{section}

\title{The subregular and submaximal p-cells}

\author{Vanessa Miemietz}
\address{V.M.: School of Engineering, Mathematics and Physics, University of East Anglia, Norwich NR4 7TJ, United Kingdom,  \newline \href{https://research-portal.uea.ac.uk/en/persons/vanessa-miemietz/}{https://research-portal.uea.ac.uk/en/persons/vanessa-miemietz/}}
\email{v.miemietz@uea.ac.uk}

\author{Marie Roth} 
\address{M.R.: School of Engineering, Mathematics and Physics, University of East Anglia, Norwich NR4 7TJ, United Kingdom}
\email{m.roth@uea.ac.uk}

\author{Daniel Tubbenhauer}
\address{D.T.: The University of Sydney, School of Mathematics and Statistics F07, Office Carslaw 827, NSW 2006, Australia, \href{http://www.dtubbenhauer.com}{www.dtubbenhauer.com}, https://orcid.org/0000-0001-7265-5047}
\email{daniel.tubbenhauer@sydney.edu.au}

\subjclass[2020]{Primary 20C08, 20C20; Secondary 14M15, 18N10.}

\keywords{Kazhdan--Lusztig cells, p-cells, p-canonical basis,
subregular cell, submaximal cell, Hecke algebras, Hecke categories,
Soergel bimodules, parity sheaves, Schubert varieties,
2-representation theory.}

\begin{document}

\begin{abstract}
	Cells for the canonical and \(p\)-canonical bases organise the representation
	theory and geometry of Hecke categories.
	We determine the relevant \(p\)-canonical basis elements and the resulting
	\(p\)-cell structure for the subregular and submaximal cells in all classical
	types.
\end{abstract}

\maketitle

\tableofcontents

\section{Introduction}

Kazhdan--Lusztig theory is one of the main bridges between the geometry of
Schubert varieties, combinatorics and representation theory. Let \(G\) be a
connected reductive complex algebraic group, let \(B\subset G\) be a Borel
subgroup, and let \(W\) be its Weyl group. The flag variety \(G/B\) is
stratified by Schubert cells
\[
	X_w^\circ= BwB/B,
	\qquad w\in W,
\]
and their closures \(X_w\) are the Schubert varieties. Over a field of
characteristic zero, the intersection cohomology complexes of the Schubert
varieties have characters given by the Kazhdan--Lusztig (KL) basis \(b_w\) of
the Hecke algebra; see \cite{KL1, KL2}.

The situation changes over a field of characteristic \(p>0\). The
decomposition theorem for maps such as Bott--Samelson resolutions can fail
with modular coefficients: torsion in the corresponding integral
intersection cohomology, or equivalently a reduction in rank of the integral intersection
forms after reduction modulo \(p\), can change the indecomposable summands. Parity
sheaves provide a robust replacement: they are designed to retain the useful decomposition-theoretic
features of intersection cohomology in this modular setting. As it turns out, for each \(w\in W\), there is an
indecomposable parity sheaf supported on \(X_w\), and its character is the
\(p\)KL basis element \(\pb_w\); see \cite{JMW} and \cite{JW1}. Thus an inequality
\[
	\pb_w\neq b_w
\]
records a genuine difference between the geometry of \(X_w\) in
characteristic \(p\) and its characteristic zero geometry.

One way to see this difference is through a Bott--Samelson resolution of
\(X_w\). The decomposition of the corresponding Bott--Samelson object is
controlled by integral local intersection forms. If the rank of one of these
forms decreases when reducing modulo \(p\), an additional indecomposable summand
appears, and hence the \(p\)KL basis differs from the KL basis. In this sense,
the coefficients of the \(p\)KL basis detect modular singularities of Schubert
varieties.

Although fairly new, the study of the \(p\)KL basis and its associated
structures has developed rapidly. Williamson's torsion explosion showed that
such modular phenomena occur for primes far larger than previously expected
and ruled out several optimistic forms of long-standing conjectures in modular
representation theory; see \cite{Wi}. At the
same time, the \(p\)KL basis has become an effective tool for modular character
theory. It occurs in character formulas for tilting modules and, through
modular Koszul duality and Smith theory, in character formulae for
representations of reductive algebraic groups; see, for example, \cite{RW1, AMRW, RW2}.

There has also been substantial progress in computation and structure.
Algorithms now make it possible to calculate large portions of the \(p\)KL
basis without performing the entire calculation directly in a combinatorial object called the diagrammatic
Hecke category; see \cite{GJW}.
Nevertheless, explicit uniform formulas remain rare and are largely confined
to particularly structured situations; see, for example, \cite{Ba, BDHN}.
At the opposite extreme, recent asymptotic results show that disagreement
between the KL and \(p\)KL bases is widespread in the classical families; see
\cite{BT}.

Thus there is a substantial gap between knowing that modular phenomena are
widespread and being able to describe their occurrences explicitly.

\subsection{Cells: Why care?}

The multiplication of canonical basis elements defines the ordinary left,
right and two-sided KL cells. Replacing the canonical basis by the \(p\)KL
basis defines the corresponding \(p\)-cells. Geometrically, these relations
measure which indecomposable parity sheaves can occur from others under
convolution. Categorically, they describe the cell structure of the geometric
or diagrammatic Hecke category. Thus \(p\)-cells record how the modular changes
in the \(p\)-canonical basis reorganize the multiplication structure of the
Hecke category.

Cells connect Hecke theory to several neighboring parts of representation
theory and geometry. At the decategorified level, they organize cell modules,
Lusztig families, special representations and primitive ideals; see, e.g., \cite{KL1, Lu1, Jo}.
Geometrically, two-sided cells also give rise to
asymptotic Hecke algebras and monoidal categories defined using truncated
convolution; see, e.g., \cite{Lu2, Be, ERT}.
Finally, cells control tensor ideals for quantum tilting modules, while
antispherical \(p\)-cells are expected to play an analogous role for tilting
modules in positive characteristic and their support varieties; see, e.g., \cite{Os1, AHR1, AHR2}.

Cell structure is also fundamental in \(2\)-representation theory. The
indecomposable objects of the Hecke category are the indecomposable
\(1\)-morphisms of a finitary \(2\)-category, and its left, right and two-sided
cells organize their behavior under composition; see \cite{MM1}.
Every simple \(2\)-representation has a unique maximal two-sided
cell which acts nontrivially, called its apex. The problem of classifying simple $2$-representations
therefore breaks naturally into separate problems, one for each
possible apex; see, for example, \cite{MM2, MMMTZ1}.
Over a field of characteristic zero, this approach led to a proof that the
Hecke categories of finite Coxeter type admit only finitely many equivalence
classes of simple \(2\)-representations, and to their classification
in every finite type except \(H_3\) and \(H_4\); see \cite{MMMTZ2}.

In positive characteristic, the indecomposable objects, their multiplication
and hence their cells may change. Before one can attempt an analogous result,
one must first understand the new two-sided \(p\)-cells which can occur.

In finite type \(A\), the \(p\)-cells are completely understood: they agree
with the ordinary KL cells for every prime and hence admit the usual
description via the Robinson--Schensted correspondence; see \cite{J1}.
The same work shows that in finite types \(B\) and
\(C\), for \(p>2\), ordinary KL cells decompose into \(p\)-cells.
Moreover, the \(p\)-canonical basis in type \(A\) is cellular; see \cite{J2}
Further structural and low-rank results appear in \cite{EJ}.
The behavior of \(p\)-cells under parabolic induction has also been studied;
see \cite{JP}.

Nevertheless, there are few uniform descriptions of \(p\)-cell
structures outside type $A$. The aim of the present paper is to give such descriptions for the
first nontrivial cells at the two ends of the ordinary two-sided cell order.

\subsection{Contribution of this article}

We study the first nontrivial cells at the two ends of the ordinary
two-sided cell order:
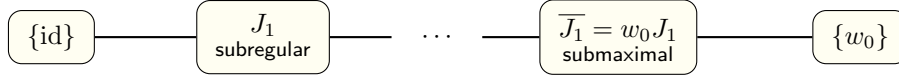
\[
	\begin{tikzpicture}[
			baseline=-0.3ex,
			every node/.style={font=\small},
			cell/.style={
					draw,
					rounded corners=5pt,
					fill=yellow!6,
					inner xsep=7pt,
					inner ysep=5pt,
					align=center
				}
		]
		\node[cell] (e) at (0,0) {\(\{\mathrm{id}\}\)};
		\node[cell] (subreg) at (2.8,0) {\(J_1\)\\[-1mm]\scriptsize subregular};
		\node at (5.1,0) {\(\cdots\)};
		\node[cell] (submax) at (7.5,0)
		{\(\Jtop=w_0J_1\)\\[-1mm]\scriptsize submaximal};
		\node[cell] (w0) at (10.7,0) {\(\{w_0\}\)};

		\draw[thick] (e) -- (subreg);
		\draw[thick] (subreg) -- (4.5,0);
		\draw[thick] (5.7,0) -- (submax);
		\draw[thick] (submax) -- (w0);
	\end{tikzpicture}
\]
Here \(J_1\) is the ordinary subregular cell, consisting of the non-identity
elements with a unique reduced expression, and \(w_0\) denotes the longest
element. The submaximal cell \(\Jtop=w_0J_1\) is its translate to the opposite
end of the cell order.

These cells are natural first test cases. The elements in the ordinary subregular cell are well-understood and have an explicit description in terms of paths in the Coxeter graph.
At the same time, they are large enough to display genuinely modular
phenomena.

Our main result is the following.

\textbf{Main result.}
In Theorems \ref{prop:subregular-dec-pcells} and \ref{thm:submax-dec-pcells}, we determine the two-sided \(p\)-cell decompositions of the subregular and
submaximal cells in all classical types \(A,B,C,D\). Except in types
\(B_n\) and \(C_n\) in characteristic two, the ordinary cells remain
single two-sided \(p\)-cells. In characteristic two they split, and the
answers for types \(B_n\) and \(C_n\) are different.

We also determine the subregular \(p\)-cells in all exceptional types and
the submaximal \(p\)-cells in types \(G_2\) and \(F_4\) (Theorem \ref{prop:exceptional-subregular-submaximal-pcells}). The exceptional
primes are \(2\) and \(3\) in type \(G_2\), and \(2\) in type \(F_4\).
In types \(E_6,E_7,E_8\), the subregular cell remains a single two-sided
\(p\)-cell for every prime.

In each case, we further record which $p$-cells contain the longest element of a parabolic subgroup.

%

As far as we know, these are the first uniform descriptions in arbitrary rank
of nontrivial two-sided \(p\)-cells outside type \(A\).
Moreover, we give more than the decomposition into two-sided \(p\)-cells.
We determine the relevant \(p\)KL basis elements in classical types explicitly in Theorems \ref{lem:pbasis_subreg} and \ref{lem:pbasis_submin}, with one isolated
submaximal element in type \(C_n\) where a complete expansion is not needed
for the cell calculation.
The exceptional cases are treated computationally in Section~5.

\subsection{The proof in a nutshell}

Let us briefly explain the proof. The first step is local and geometric. We
compute the intersection forms controlling the decomposition of selected
Bott--Samelson objects in small rank. The only reductions in rank in the classical families occur
at the exceptional edge and in characteristic two. These calculations give
the small rank results from which all remaining formulae are obtained, with the exception of one case in type $C_n$.

The second step is combinatorial. Which canonical basis element can occur as
a lower term of a \(p\)KL basis element is severely constrained by the Bruhat order in the subregular case and the left and right descent sets in the submaximal case. These
restrictions leave very few possibilities, which turns what could be a large
calculation in the Hecke category into a small list of local cases.

The third step propagates the local information through the Coxeter graph.
Inversion exchanges left and right calculations. Star operations move the
rank two relations along the simply laced part of the graph, while induction
from parabolic subgroups transports calculations from lower rank to arbitrary
rank. The proof can be summarized as follows:
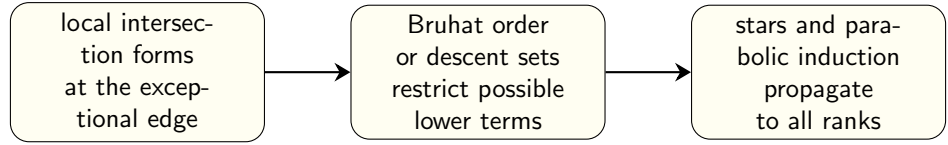
\[
	\begin{tikzpicture}[
		baseline=-0.5ex,
		every node/.style={font=\small},
		box/.style={
				draw,
				rounded corners=7pt,
				fill=yellow!6,
				text width=3.1cm,
				minimum height=1.2cm,
				align=center
			},
		arrow/.style={-{Stealth[length=6pt,width=7pt]},thick}
		]
		\node[box] (local) at (0,0)
		{local intersection forms\\at the exceptional edge};
		\node[box] (restrict) at (4.5,0)
		{Bruhat order \\ or descent sets \\restrict possible\\ lower terms};
		\node[box] (global) at (9,0)
		{stars and parabolic induction\\propagate to all ranks};

		\draw[arrow] (local) -- (restrict);
		\draw[arrow] (restrict) -- (global);
	\end{tikzpicture}
\]

Once the necessary \(p\)KL expansions are known, the \(p\)-cell relations are obtained from those in small rank by using star operations and parabolic induction, and counting the intersections
of left and right cells determines the final cell shapes. Thus the
calculation is mostly uniform in rank and ultimately rests on a small amount of
local geometry.
Exceptional types are treated computationally.

\subsection{Structure of this article}

The paper is organized as follows. In Section~2 we recall the
\(p\)KL basis and collect the tools used throughout the paper.
Section~3 treats the subregular cell in classical type, first determining
the relevant \(p\)KL basis elements and then the resulting \(p\)-cells.
Section~4 follows the same program for the submaximal cell.
Finally, Section~5 contains the calculations in exceptional type.

\noindent\textbf{Acknowledgments.}
VM is partially supported by EPSRC grant EP/Z533750/1. MR is supported by EPSRC grant EP/Z533750/1.
DT thanks the computer for checking the exceptional types without asking to
be listed as a fourth author or why this would be useful in any way or form, and acknowledges support from ARC Future
Fellowship FT230100489.

\section{Background}

\subsection{The $p$-canonical basis and $p$-cells}
\subsubsection{Coxeter systems}

Let $(W,S)$ be a \textit{Coxeter system}. In other words, there is a matrix $(m_{s,t})_{s,t \in S}$ with $m_{s,s} =1$ and $m_{s,t} = m_{t,s} \in \{2, 3,\dots\} \cup \{\infty\}$ for $s \neq t$ such that \[W = \langle s \in S \mid (st)^{m_{s,t}} =1 \rangle.\]
Here we will assume that $W$ is finite and $m_{s,t}\in \{1,2,3,4,6\}$ for $s,t \in S$, that is, $(W,S)$ is a crystallographic Coxeter system.

We also fix a Cartan realisation $\mathfrak{h}$  of $(W,S)$ over $\bbZ$ (see \cite[Section 10.1]{AMRW}). In particular, it is a free $\bbZ$-module of finite rank  with an action of $W$ given on the \textit{simple co-roots} $\{\alpha_s^\vee\}_{s\in S}$ by \[t.\alpha_s^\vee = \alpha_s^\vee - \alpha_t(\alpha_s^\vee)\alpha^\vee_t\] for $t\in S$. Here $\alpha_t \in \Hom(\mathfrak{h},\bbZ)$ is a \textit{simple root}.

To $(W,S)$ one associates a \textit{Dynkin diagram} as follows. Its vertices are labelled by $S$. For $s,t\in S$ with $s\neq t$,
the corresponding vertices are joined by $|\alpha_s(\alpha_t^\vee)|$ edges if $|\alpha_s(\alpha_t^\vee)| \geq |\alpha_t(\alpha_s^\vee)|$. If moreover, $|\alpha_s(\alpha_t^\vee)| > 1$, the edge is oriented towards the vertex labelled $s$.

We use the following labelling conventions for Dynkin diagrams in classical types:
\begin{gather*}
	\rA_n:\ \dynA
	\qquad
	\rB_n:\ \dynC
	\\
	\rC_n:\ \dynB
	\qquad
	\rD_n:\ \dynD.
\end{gather*}
The displayed diagrams are schematic: in type \(\rA_n,\rB_n,\rC_n\) the
vertices are \(1,\ldots,n\) from left to right. In types \(\rB_n\) and
\(\rC_n\), the exceptional edge is between \(s_1\) and \(s_2\). In type
\(\rD_n\), the fork is attached at \(s_3\).

For exceptional types, we use the following labelling convention:
\[
	G_2:\ \dynGtwo
	\qquad
	F_4:\ \dynFfour
\]
\[
	E_6:\ \dynEsix
	\qquad
	E_7:\ \dynEseven
	\qquad
	E_8:\ \dynEeight.
\]

\subsubsection{Expressions, decoration, defect}
The elements in $S$ are called \textit{simple reflections} and an \textit{expression} $(s_1,\dots,s_n)$ is a sequence of elements in $S$. The product $s_1\cdots s_n$ gives an element $w \in W$ and we will write $\underline{w} = (s_1,\dots,s_n)$ (or sometimes $\underline{w}=s_1\cdots s_n$). We say that the expression is \textit{reduced} if it has minimal length amongst all expressions for $w$. The length of a reduced expression defines the \textit{length} $l(w)$ of the element $w$.

\begin{conv}[Shortest path]\label{not:shortest_path}
	Let \(W\) be of classical type, with simple reflections
	\(S=\{s_1,\dots,s_n\}\). If \(a,b\) are vertices of the Coxeter graph
	and there is a unique shortest path from \(a\) to \(b\), write
	\[
		[a,b]
	\]
	for the product of the simple reflections along this path, read from
	\(a\) to \(b\). Thus, in type \(A_n\),
	\[
		[a,b]=
		\begin{cases}
			s_a s_{a+1}\cdots s_b, & a\leq b, \\
			s_a s_{a-1}\cdots s_b, & a\geq b.
		\end{cases}
	\]
	The same notation is used in type \(D_n\), on the forked Coxeter graph
	with the fork attached at \(s_3\).
\end{conv}

\begin{conv}[Exceptional truncations]\label{not:truncations}
	Assume \(W\) is of type \(B_n\) or \(C_n\), with exceptional edge between
	\(s_{1}\) and \(s_2\). For \(a,b\geq 2\), write
	\[
		[a,1,b]=s_a s_{a-1}\cdots s_{2}\,s_1\,s_{2}\cdots s_b.
	\]
	For \(\min(a,b)\geq k\geq 1\), define the \(k\)-th truncation of this
	element by
	\[
		[a,k,b]
		:=
		s_a s_{a-1}\cdots s_{k+1}\,s_k\,s_{k+1}\cdots s_b.
	\]
	Thus \([a,1,b]\) is the original word, while larger \(k\) means that the
	word turns around earlier, before reaching the exceptional edge.
\end{conv}

For $\underline{w} = (s_1,\dots,s_n)$ an expression, a \textit{subexpression} of $\underline{w}$ is a sequence $e = (e_i)_{1 \leq i \leq n} \in \{0,1\}^n$. If $\underline{w}$ is a reduced expression for $w$, then we write that $x = s_1^{e_1} \cdots s_n^{e_n} \leq w$. This defines the \textit{Bruhat order} on $W$, see, e.g., \cite[Section 1.2.6]{EMTW} for more details.

To a subexpression $e$ of $\underline{w}$, one associates a \textit{decoration} $d = (d_i)_{1 \leq i \leq n} \in \{U,D\}^n$ by setting \[d_1 = U \text{ and } d_i = \begin{cases}
		U & \text{if } s_1^{e_1}\cdots s_{i-1}^{e_{i-1}}s_i>s_1^{e_1}\cdots s_{i-1}^{e_{i-1}} \\
		D & \text{otherwise}
	\end{cases} \text{ for } 2 \leq i \leq n.\]

The defect of $e$ is given by \[\mathrm{df}(e) := |\{1 \leq i \leq n \mid d_ie_i=U0\}|-|\{1 \leq i \leq n \mid d_ie_i=D0\}|.\]

\begin{example}\label{ex:defect}
	Assume that $W$ is of type $B_6$.

	Let $\underline{w}= s_6s_5s_4s_3s_2s_1s_2s_3s_4s_5s_6 = [6,1,6]$ and $e=11101010101$. It corresponds to the element $s_6s_5s_4s_2s_2s_4s_6=s_5s_6s_5$. Its decoration is \[U1U1U1U0U1U0D1U0D1D0U1\] and $\mathrm{df}(e) =2$.
\end{example}

\subsubsection{Hecke algebra}
To $(W,S)$ we associate the \textit{Hecke algebra} $\rH$, which is free over $\bbZ[v,v^{-1}]$ and is generated  by $\{\delta_w \mid w \in W\}$, subject to relations for $s,t\in S$ \begin{align*}
	\delta_s\delta_s                                     & = (v^{-1}-v)\delta_s +1,                                \\
	\underbrace{\delta_s\delta_t\delta_s\dots}_{m_{s,t}} & = \underbrace{\delta_t\delta_s\delta_t\dots}_{m_{s,t}}.
\end{align*}
For $w \in W$ with reduced expression $(s_1, \dots, s_n)$, we set $\delta_w := \delta_{s_1}\cdots\delta_{s_n}$. The set $\{\delta_w \mid w \in W\}$ is the standard basis.

There is a unique $\bbZ$-linear involution $\bar{(-)}$ on $\rH$ satisfying $\bar{v} = v^{-1}$ and~$\bar{\delta_x} = \delta_{x^-1}^{-1}$.

We fix \[b_s = \delta_s+v.\] Observe that $b_s=\bar{b_s}$. For an expression $\underline{w}=(s_1,\dots,s_n)$, we write~$b_{\underline{w}} = b_{s_1}\cdots b_{s_n}$.

\subsubsection{Hecke category and $p$-canonical basis}
Let $O$ be any complete local ring. Consider $\mathfrak{h}_O = \mathfrak{h} \otimes_\bbZ O$ and assume it satisfies Demazure surjectivity. The main reference is \cite{EW}, more precisely Lemma 6.24, Theorem 6.25, and Corollary 6.26.

Let $\H$ be the diagrammatic Hecke category associated with $\mathfrak{h}_O$.
\begin{theorem}[Properties of $\H$]
	Let $O$ be a complete, local, integral domain. Then \begin{enumerate}
		\item $\H$ is a Krull-Schmidt category.
		\item For all $w \in W$ there exists a unique indecomposable object $^OB_w \in \H$ which is a direct summand of $B_{\underline{w}}$ for any reduced expression $\underline{w}$ of $w$ and which is not isomorphic to a grading shift of any direct summand of any expression $B_{\underline{v}}$ for $v < w$. In particular, the object $^OB_w$ does not depend up to isomorphism on the reduced expression $\underline{w}$ of $w$.
		\item The set $\{^OB_w \mid w \in W\}$ gives a complete set of representatives of the isomorphism classes of indecomposable objects in $\H$ up to grading shift.
		\item There exists a unique isomorphism of $\bbZ[v,v^{-1}]$-algebras	$\mathrm{ch} : [\H] \to \rH$ sending $[^OB_s]$ to $b_s$ for all $s \in S$, where $[\H]$ denotes the split Grothendieck group of $\H$.
	\end{enumerate}
\end{theorem}

Note that in particular, we have $\mathrm{ch}(B^O_{\underline{w}}) = b_{\underline{w}}$ for any $w \in W$. We are now ready to define the $p$-canonical basis as in \cite[Definition 3.2]{JW1}.

\begin{definition}[$p$-canonical basis]
	Let $\Bbbk$ be a field of positive characteristic $p >0$. The \textit{$p$-canonical basis} of $\rH$ is $\{\pb_w \mid w \in W\}$ where \[\pb_w := \mathrm{ch}([^\Bbbk B_w]).\]

	If $\BbbK$ is a field of characteristic zero, the map $\mathrm{ch}$ yields the canonical basis $\{b_w \mid w\in W\}$.
\end{definition}

\begin{remark}The $p$-canonical basis depends only on the characteristic of $\Bbbk$ and of the choice of the realisation $\mathfrak{h}$.
\end{remark}

The map $\mathrm{ch}$ yields the various coefficients. We fix the following notation for $x \in W$ with reduced expression $\underline{x}$. We write
\[b_{\underline{x}} = b_x + \sum_{z<x}h_{z,x}b_z, \quad \text{and} \quad b_{\underline{x}} = \pb_x + \sum_{z<x}\leftidx{^p}{h}{_{z,x}}\pb_z\]
and \[\pb_x = b_x + \sum_{y<x}\pmk_{y,x}b_y.\]
In particular, we have
\begin{equation}
	\pb_{x} = b_x + \sum_{z<x}b_z(h_{z,x}-\ph_{z,x}-\sum_{z<y<x}\pmk_{z,y}\ph_{y,x}).
\end{equation}

By \cite[Theorem 4.2]{JW1}, $h_{z,x},{^p}{h}{_{z,x}},\pmk_{z,x} \in \bbZ_{\geq 0}[v,v^{-1}]$.

Let $x,y \in W$. We write
\[b_xb_y = \sum_{z \in W}\mu_{x,y}^zb_z \quad \text{and} \quad \pb_x\pb_y = \sum_{z \in W}\pmu_{x,y}^z\pb_z.\]
For $z \in W$, the coefficient of $b_z$ in $\pb_x\pb_y$ is then given by
\begin{equation}\label{eq:pmuxyz}
	\sum_{v,w}\pmk_{v,x}\pmk_{w,y}\mu_{v,w}^z = \pmu_{x,y}^z + \sum_{w>z}\pmk_{z,w}\pmu_{x,y}^w.
\end{equation}

We say that $x \leq_L^p y$ (resp. ~ $x \leq_R^p y$) if there is $w \in W$ such that $\pmu_{w,x}^y \neq 0$ (resp.~ $\pmu_{x,w}^y \neq 0$). Note that this is the opposite preorder to that given in \cite[Definition 3.1]{J1}. The partial preorder $\leq_J^p$ is
generated by $\leq_L^p$ and $\leq_R^p$, i.e.~$x \leq_J^p y$ if there is $w,w'\in W$ such that $\pb_y$ appears as a summand of $\pb_w\pb_x\pb_{w'}$.
The equivalence classes for each of the partial preorders are called left, right or two-sided $p$-cells.

If $p=0$, we will drop the superscript $0$ and write $\leq_J$ instead of $\leq_J^0$ for instance.

\subsection{Useful tools to compute the $p$-canonical basis}
We collect various useful results to compute the $p$-canonical basis that we will apply later on.

\subsubsection{Local intersection form}\label{subsubsec:localintform} There is a general method to compute the coefficients ${^p}{h}{_{y,x}}$ using only the information at the level of $B_{\underline{x}}$. We give a quick overview, referring the reader to \cite[Section 5]{GJW} or \cite[Section 27.3]{EMTW} for more details.

Let $\underline{x},\underline{y}$ be two reduced expressions in $W$ such that $y \leq x$ (in the Bruhat order). Let $\Hom^\bullet_{< y}(B_{\underline{y}},B_{\underline{x}})$ the set of morphisms in the graded space $\Hom^\bullet(B_{\underline{y}},B_{\underline{x}})$ that factors through some $B_{\underline{w}}$ for $\underline{w}$ a reduced expression of $w < y$.
Set $\Hom^\bullet_{\not< y}(B_{\underline{y}},B_{\underline{x}}) := \Hom^\bullet(B_{\underline{y}},B_{\underline{x}})/\Hom^\bullet_{< y}(B_{\underline{y}},B_{\underline{x}}).$

For any $d \in  \bbZ$, the $d$-th graded piece of the local
intersection form is the restriction\[I^d_{\underline{x},\underline{y}}: \Hom^d_{\not< y}(B_{\underline{y}},B_{\underline{x}}) \times \Hom^{-d}_{\not< y}(B_{\underline{x}},B_{\underline{y}}) \to \bbZ.\]

We then have \[{^p}{h}{_{y,x}} = \sum_{d\in \bbZ} \mathrm{rank}(I^d_{\underline{x},\underline{y}} \otimes_O \Bbbk)v^d.\]

\begin{remark}
	The spaces $\Hom^\bullet_{\not< y}(B_{\underline{y}},B_{\underline{x}})$ and $\Hom^\bullet_{\not< y}(B_{\underline{x}},B_{\underline{y}})$ have basis given in terms of \textit{light leaves}. Let $e$ be a subexpression of $\underline{x}$ such that $\underline{x}^e = y$. A light leaf is a particular morphism from $B_{\underline{x}}$ to $B_{\underline{x}^e}$, which can be given in terms of diagram, using the decoration of $e$. Its degree is given by the defect of $e$. For more information on how to compute light leaves, see \cite[Section 10.4]{EMTW}.
\end{remark}



\subsubsection{Descent sets}
One first tool to control the decomposition of the $p$-canonical basis element $\pb_x$ into $b_y$ is to use the descent sets of $x$.

\begin{conv}[Descent sets]\label{conv:descentset}
	Let $w\in W$. The \textit{left descent set} $L(w)$ of $w$ is defined as follows:
	\[L(w) := \{s \in S \mid l(sw)<l(w)\}.\]
	The \textit{right descent set} $R(w)$ is defined symmetrically as\[R(w) := \{s\in S \mid l(ws)<l(w)\}.\]
\end{conv}
Note that $L(w) = R(w^{-1})$. Moreover, for $x,y \in W$,\begin{align*}
	x \leq_L^p y & \implies R(x) \subseteq R(y)  \\
	x \leq_R^p y & \implies L(x) \subseteq L(y).
\end{align*}

\begin{lemma}[{\cite[Proposition 4.2]{JW1}}]\label{lem:myx_descentset} Let $x,y \in W$. Then if $\pmk_{y,x} \neq 0$ then $y \leq x$, $L(x) \subseteq L(y)$ and $R(x)\subseteq R(y)$.
\end{lemma}

\begin{remark}\label{rmk:cantbeid}
	Let $x \in W\setminus\{id\}$ and $y \in W$. If $\pmk_{y,x} \neq 0$, then we have $\emptyset \neq L(x) \subseteq L(y)$, whence $y \neq id$.
\end{remark}

\begin{conv}[Longest element]\label{not:longest_elt} We denote by $w_0$ the longest element of~$W$. In type~$B_n$ (or~$C_n$), we have \[w_0 = s_1s_2s_1s_2[3,1,3]\dots [n,1,n].\]
	In type~$B_n$ and~$C_n$, the element~$w_0$ is central (\cite[Remark 13.1.8]{D}). Note that \[[i,1,i][j,1,j]=[j,1,j][i,1,i] \text{ for } 2 \leq i,j \leq n.\]
\end{conv}

\begin{remark}\label{rmk:pbw0}
	Let $x \in W$ such that $\pmk_{x,w_0} \neq 0$. Then $L(x)=R(x)= S$, whence $x = w_0$.
\end{remark}

\begin{lemma}[{\cite[Exercise 2.10]{BB1}}]\label{lem:w0DescentSet}
	Let $w \in W$. Then
	\begin{enumerate}
		\item $L(ww_0) = S \backslash L(w)$ and $L(w_0w) = S \backslash w_0L(w)w_0$
		\item $R(ww_0) = S \backslash w_0R(w)w_0$ and $R(w_0w) = S \backslash R(w)$.
	\end{enumerate}
\end{lemma}

\begin{remark}\label{rmk:myx_zero}
	Let $x,y \in W$. Let $s \in L(x)$. If $b_y$ does not appear in the decomposition of $b_s\pb_{sx}$, then $\pmk_{y,x} =0$. Symmetrically, for  $s \in R(x)$, if $b_y$ does not appear in the decomposition of $\pb_{xs}b_s$, then $\pmk_{y,x} =0$. Indeed,
	\[b_s\pb_{sx} = b_sb_{sx} + \sum_{z<sx}\pmk_{z,sx}b_sb_z = b_x + \sum_{y<sx}\mu_{s,sx}^yb_y + \sum_{y<sx}\pmk_{z,sx}\left(\sum_{y \leq z}\mu_{s,z}^y\right)b_y.\]
	In particular, $\pmu_{s,sx}^x =1$ and if $\pmk_{y,x} \neq 0$, then $b_y$ appears in the decomposition of $b_s\pb_{sx}$.
\end{remark}

\subsubsection{The star operation and inverse}
A second approach is to use various involutions $\iota$ on $W$ such as the star operation or the inverse map and see how the coefficients $\pmk_{y,x}$ and $\pmk_{\iota(y),\iota(x)}$ are related.

\begin{lemma}[{\cite[Proposition 4.2]{JW1}}]\label{lem:myx_inverse} Let $x,y \in W$, then
	\[\pmk_{y,x} = \pmk_{y^{-1},x^{-1}} \quad \text{and} \quad \ph_{y,x} = \ph_{y^{-1},x^{-1}}.\]
\end{lemma}

\begin{definition}[Star operation]\label{def:star_operation}
	Let $s,t \in S$ such that $sts=tst$. The \textit{left star operation} $*_{s,t}$ \textit{with respect to $s$ and $t$} is defined as follows. Let $w \in W$ such that $|\{s,t\} \cap L(w)|= 1$. Then $\starop{s}{t}{w}$ is the unique element in $\{sw,tw\}$ such that $|\{s,t\} \cap L(\starop{s}{t}{w})|= 1$. Analogously, we define the \textit{right star operation} on the right for the right descent set.
\end{definition}
\begin{remark}
	Let $s,t,w\in W$ as in the previous definition with $s \in L(w)$. In particular, there is $z \in W$ such that $w=sz$ for $z \in W$ with $l(z) = l(w)-1$. Then the following two cases can occur.\begin{enumerate}
		\item $t \not \in L(z)$. In this case, $\starop{s}{t}{w} = tw =tsz$.
		\item $t \in L(z)$. Then there is $z'\in W$ such that $w=stz'$ and $l(z') = l(w)-2$. Note that $s \not \in L(z')$ since $sts=tst$ and $t \not \in L(w)$. In this case, $\starop{s}{t}{w} = tz'$.
	\end{enumerate}
\end{remark}
\begin{lemma}[{\cite[Corollary 4.7]{J1}}]\label{lem:myx_starop}
	Let $s,t \in S$ such that $sts=tst$. For $x,y \in W$ with $|L(x)\cap \{s,t\}| = |L(y)\cap \{s,t\}| = 1$, we have \[\pmk_{y,x} = \pmk_{\starop{s}{t}{y},\starop{s}{t}{x}}.\]
\end{lemma}

The star operation commutes with multiplication by $w_0$.
\begin{lemma}\label{lem:starop_w0}
	Let $w \in W$ and $s,t \in S$ such that $sts=tst$ and, moreover,  $|\{s,t\} \cap L(w)|=~1$. Then \[\starop{s}{t}{(ww_0)}= \starop{s}{t}{w}w_0.\]
	The same result holds for the right star operation.
	If, moreover, $W$ is of type $B_n$ or $C_n$, then
	\[\starop{s}{t}{(w_0w)}= w_0\starop{s}{t}{w}\]
\end{lemma}
\begin{proof}
	This is a consequence of Lemma \ref{lem:w0DescentSet}. Indeed, we may assume that $s \in L(w)$ and $t \not \in L(w)$, then $s \not \in L(ww_0)$ and $t \in L(ww_0)$ and we can apply the star operation. By Lemma \ref{lem:w0DescentSet}, $s \in L(\starop{s}{t}{w}w_0)$ and $t \not \in L(\starop{s}{t}{w}w_0)$. Thus, $\starop{s}{t}{(ww_0)}= \starop{s}{t}{w}w_0$.
\end{proof}

Lastly, the star operation works well with left/right order as explained in \cite[Corollary 4.12 and Theorem 4.13]{J1}.
\begin{lemma}\label{lem:starop_cell}
	Let $x,y \in W$.  Let $s,t \in S$ such that $sts=tst$ and further $|\{s,t\} \cap L(x)|=~1 = |\{s,t\} \cap L(y)|$. Then
	\begin{enumerate}
		\item $x \sim_L^p \starop{s}{t}{x}$,
		\item if $x \leq_R^p y$, then $\starop{s}{t}{x} \leq^p_R \starop{s}{t}{y}$ whence, if  $x \sim_R^p y$, then $\starop{s}{t}{x} \sim^p_R \starop{s}{t}{y}$
	\end{enumerate}
	The same result holds for the right star operation.
\end{lemma}

\subsubsection{Parabolic subgroups}

The last tool we will use is induction from parabolic subgroups of $W$.
\begin{definition}
	Let $I \subseteq S$. We write $W_I$ for the \textit{parabolic subgroup} of $W$ generated by the simple reflections in $I$. We denote by $W^{I}$ the set of minimal length of coset representatives in $W/W_I$.
\end{definition}

\begin{lemma}[{\cite[Lemma 3.12]{J1}}]\label{lem:redexprParabolic} Let $I \subseteq S$ and $w \in W^I$ with reduced expression $\underline{w}$ and $x \in W_I$ with a fixed reduced expression $\underline{x}$. Let $y \in W_I$, then there is a decoration and defect preserving bijection:
	\begin{align*}
		\{\text{subexpression of } \underline{x} \text{ for } y\} & \to \{\text{subexpression of } \underline{w}\underline{x} \text{ for }wy\} \\
		e                                                         & \mapsto \underline{w}e
	\end{align*}

\end{lemma}
\begin{lemma}[{\cite[Theorem 3.9]{J1}}]\label{lem:order_parabolic}
	Let $I \subseteq S$ and $x,y \in W_I$ such that $x \leq_R^p y$ in $W_I$ if and only if for all $w \in W^I$, \[wx \leq_R^p wy \] in W.
\end{lemma}

\begin{lemma}[{\cite[Corollary 3.10]{J1}}]\label{lem:hyx_parabolic}
	Let $I \subseteq S$ and $w \in W^I$ with reduced expression $\underline{w}$ and $x \in W_I$ with a fixed reduced expression $\underline{x}$. Let $y \leq x$. Then
	\[\ph_{wy,wx} = \ph_{y,x} \text{ and }  h_{y,x} = h_{wy,wx}.\]
\end{lemma}

\begin{lemma}\label{lem:myx_parabolic}
	Let $I \subseteq S$ and $w \in W^I$ with reduced expression $\underline{w}$ and $x \in W_I$ with a fixed reduced expression $\underline{x}$. Let $y \leq x$. Then
	\[ \pmk_{wy,wx} = \pmk_{y,x}.\]
\end{lemma}
\begin{proof}
	We proceed by induction on $l(x)-l(y)$. If $l(x)-l(y) =0$, then $y=x$ and the result is clear.
	Assume now that $l(x)-l(y) = k$ and that for all $y' \in W_I$ such that $l(x)-l(y')<k$, $\pmk_{wy',wx} = \pmk_{y',x}$.

	Recall that
	\[\pmk_{y,x} = h_{y,x}-\ph_{y,x}-\sum_{y<z<x}\pmk_{y,z}\ph_{z,x}.\]

	If $z \in W$ such that $y<z<x$, then $z \in W_I$ and $wy<wz<wx$. On the other hand, if $z \in W$ with $wy<z<wx$, then $z \in wW_I$ by \cite[Proposition 2.5.1]{BB1}. Hence, there is $z'\in W_I$ such that $z =wz'$.
	Thus, there is a bijection
	\begin{align*}
		\{z \in W_I \mid y<z<x\} & \to \{z \in W \mid wy<z<wx\} \\
		z                        & \mapsto wz.
	\end{align*}
	Therefore, by Lemma \ref{lem:hyx_parabolic} and using the induction hypothesis,
	\begin{align*}
		\pmk_{y,x} & = h_{y,x}-\ph_{y,x}-\sum_{y<z<x}\pmk_{y,z}\ph_{z,x}                         \\
		           & = h_{wy,wx}-\ph_{wy,wx}-\sum_{wy<wz<wx}\pmk_{wy,wz}\ph_{wz,wx}              \\
		           & = h_{wy,wx}-\ph_{wy,wx}-\sum_{wy<z<wx}\pmk_{wy,z}\ph_{z,wx} = \pmk_{wy,wx}.
	\end{align*}

\end{proof}

\section{The subregular cell in classical type}

We will now describe the $p$-cells for the subregular two-sided $0$-cell in classical types.

\begin{conv}[The subregular cell]\label{not:subregular} We write \(J_1\) for the ordinary
	subregular two-sided cell, i.e. the set of non-identity elements with a
	unique reduced expression.

	In types \(A_n\) and \(D_n\), every element of \(J_1\) is an oriented
	path:
	\[
		J_1=\{[a,b]\mid 1\leq a,b\leq n\}.
	\]
	The left cell with right descent \(s_i\), and the right cell with left
	descent \(s_i\), are
	\[
		L_{s_i}=\{[a,i]\mid 1\leq a\leq n\},
		\qquad
		R_{s_i}=\{[i,a]\mid 1\leq a\leq n\}.
	\]
	In particular,
	\[
		L_{s_i}\cap R_{s_i}=\{s_i\}.
	\]

	In types \(B_n\) and \(C_n\), the elements of \(J_1\) are of two kinds.
	First, there are the path elements
	\[
		[a,b]\qquad (1\leq a,b\leq n).
	\]
	Second, for \(a,b\geq 2\), there are the elements which go down to the
	exceptional edge, turn around, and then go back up: \[[a,1,b] \qquad (2\leq a,b\leq n).\]

	There is also the element $[1,2,1].$
	Thus
	\[
		J_1=
		\{[a,b]\mid 1\leq a,b\leq n\}
		\sqcup
		\{[a,1,b]\mid 2\leq a,b\leq n\}
		\sqcup
		\{[1,2,1]\}.
	\]

	With the convention that \(L_{s_i}\) denotes the left cell with right
	descent \(s_i\), one has
	\[
		L_{s_1}
		=
		\{[a,1]\mid 1\leq a\leq n\}\sqcup\{[1,2,1]\},
	\]
	and, for \(i\geq 2\),
	\[
		L_{s_i}
		=
		\{[a,i]\mid 1\leq a\leq n\}
		\sqcup
		\{[a,1,i]\mid 2\leq a\leq n\}.
	\]
	Similarly, \(R_{s_i}\) is obtained by reversing the words. Hence
	\[
		|L_{s_1}|=n+1,
		\qquad
		|L_{s_i}|=2n-1\quad (i\geq 2),
	\]
	and
	\[
		L_{s_1}\cap R_{s_1}=\{s_1,\ s_1s_{2}s_1\},
	\]
	while, for \(i\geq 2\),
	\[
		L_{s_i}\cap R_{s_i}
		=
		\{s_i,\ [i,1,i]\}.
	\]
\end{conv}

\begin{remark}
	The notation is chosen so that the first index records the left descent
	and the last index records the right descent. Thus, e.g.,
	\(L_{s_i}\) is indexed by the right endpoint \(i\).
\end{remark}
\subsection{The $p$-Kazhdan--Lusztig elements}
We first describe the decomposition of $\pb_x$ for $x \in J_1$.

Thanks to the star operation, we inductively show that if $x \in J_1$ is of the form $[a,b]$, then $\pb_x=b_x$.
\begin{lemma}\label{lem:pbx_shortestpath} Let $p$ be any prime number.
	Let $1\leq a,b,\leq n$ and $x = [a,b]$ be the shortest path. Then
	\[\pb_x=b_x. \]
\end{lemma}
\begin{proof}
	We proceed by induction on $l(x)$. If $l(x) \leq 2$, then we know that $\pb_x=b_x$.

	Assume now that $l(x) = l>2$ and that for all $x'$ of the form $[a',b']$ for some $1 \leq a', b'\leq n$ with $l(x')<l$, $\pb_{x'}= b_{x'}$.

	Let $y \leq x$ such that $\pmk_{y,x} \neq 0$. Suppose first that $a>b$. In particular, $a \geq 3$. By Lemma \ref{lem:myx_descentset}, $s_a \in L(y)$. Moreover, $s_{a-1} \not \in L(y)$ since $y$ is less than $x$ in the Bruhat order. Thus, we can apply the left star operation to $x$ and $y$ with respect to $\{s_a,s_{a-1}\}$ and obtain ${}^*x$ and ${}^*y$. Note that ${}^*x =[a-1,b]$. By Lemma \ref{lem:myx_starop} and the induction hypothesis,
	\[\pmk_{y,x} = \pmk_{{}^*y,{}^*x} = \begin{cases}
			1 \text{ if }{}^*y={}^*x, \\
			0 \text{ otherwise.}
		\end{cases}\]
	Therefore, $y=x$ and $\pb_x = b_x$.

	Suppose now that $b>a$. By the previous case, since $x^{-1}=[b,a]$, $\pb_{x^{-1}}=b_{x^{-1}}$. Thus, by Lemma \ref{lem:myx_inverse}, $\pb_x=b_x$.
\end{proof}

We can now assume that $W$ is of type $B_n$ or $C_n$ and that $x$ is of the form $[a,1,b]$ for $1 \leq a,b \leq n$. The description of the decomposition (Theorem \ref{lem:pbasis_subreg}) will be proven by induction, the base case being when $n=2$. We first describe the elements $z<x$ whose left or right descent sets satisfies certain properties.
The following lemma provides results used towards the induction step, considering the case $n \in \{a,b\}$.

\begin{lemma}\label{lem:techlem1}
	Assume that $W$ is of type $B_n$ or $C_n$. Let $2 \leq a,b \leq n$.
	\begin{enumerate}
		\item Let $x = [n,1,b]$ and $z<x$ with $s_n \in L(z) \subseteq \{s_n,s_{n-1}\}$. Then, there exists $z_1<[n-1,1,b]$ with  $z=s_nz_1$ and $l(z) = 1+l(z_1)$. Moreover, $L(z_1) \subseteq \{s_{n-1}\}$.
		\item  Let $x = [a,1,n]$ and $z<x$ with $s_n \in R(z) \subseteq \{s_n,s_{n-1}\}$. Then, there exists $z_1<[a,1,n-1]$ with  $z=z_1s_n$ and $l(z) = 1+l(z_1)$. Moreover, $R(z_1) \subseteq \{s_{n-1}\}$.
		\item Let $x = [n,1,n]$ and $z<x$ with $s_n \in L(z) \subseteq \{s_n,s_{n-1}\}$ and $s_n \in R(z) \subseteq \{s_n,s_{n-1}\}$. Then, either $z =s_n$ or there is $z_1<[n-1,1,n-1]$ with  $z=s_nz_1s_n$ and $l(z) = 2+l(z_1)$. Moreover, $L(z_1) = \{s_{n-1}\}$ and $R(z_1) = \{s_{n-1}\}$.
	\end{enumerate}
\end{lemma}
\begin{proof} We first prove (1), the proof of (2) being entirely symmetric.

	By \cite[Proposition 2.5.1]{BB1}, there is $z_1<[n-1,1,b]$ and $z_0 \in W_{\{s_n\}}$ such that $z = z_0z_1$ and $l(z) = l(z_0)+l(z_1)$. Assume for a contradiction that $z_0$ is the identity, i.e. $z<[n-1,1,b]$. Since $s_n \in L(z)$, we must have $z \in W_{\{s_1,\dots,s_{n-2}\}}s_n$ (and $b=n$). If $z \neq s_n$, there is $1 \leq i \leq n-2$, such that $s_i \in L(z)$, which contradicts $L(z) \subseteq \{s_n,s_{n-1}\}$. Thus, $z=s_nz_1$ with $z_1<[n-1,1,b]$ and  $l(z) = 1+l(z_1)$.
	Moreover, $L(z_1) \subseteq \{s_{n-1}\}$ since $L(z) = \{s_n,s_{n-1}\}$.

	We now check (3), which is a corollary of (1) and (2). Assume that $z \neq s_n$. By (1), there exists $id \neq z_2<[n-1,1,n]$ with  $z=s_nz_2$ with $l(z) = 1+l(z_2)>1$. Moreover, $L(z_2) = \{s_{n-1}\}$. Observe that $R(z_2) \subseteq \{s_n,s_{n-1}\}$. If $s_n \not \in R(z_2)$, then $z_2<[n-1,1,n-1]$. Now that,  since $s_n \in R(z)$, we must in fact have $z_2 \in W_{\{s_1,\dots,s_{n-2}\}}$. Since $z_2\neq id$, this contradicts $R(z_2)\subseteq \{s_n,s_{n-1}\}$. Thus, $s_n \in R(z_2)$ and we may apply (2) to $z_2<[n-1,1,n]$. There exists $z_1<[n-1,1,n-1]$ with $z_2 = z_1s_n$ and $l(z_2)=1+l(z_1)$. Since $L(z_2) = \{s_{n-1}\}$, $z_1$ is not the identity. Moreover, $R(z_1) = \{s_{n-1}\}$. Therefore, $z=s_nz_1s_n$ for $z_1<[n-1,1,n-1]$ and $L(z_1)=R(z_1)=\{s_{n-1}\}$.
\end{proof}

\begin{lemma}\label{lem:subexpressionJ1}
	Assume that $W$ is of type $B_n$ or $C_n$. Let $2 \leq a,b \leq n$ and $x = [a,1,b]$. Let $z\leq x$ such that $L(z) = \{s_a\}$ and $R(z)=\{s_b\}$. Then
	\begin{itemize}
		\item either $z = [a,b]$
		\item or $z = [a,1,b]$.
	\end{itemize}
\end{lemma}
\begin{proof}
	If $n=2$, then it is clear as $x = s_2s_1s_2$. Assume now that the result holds for $B_k$ with $k<n$. In particular, if $a,b<n$, then $z \in \{[a,b],[a,1,b]\}$. Thus, we may assume that $n \in \{a,b\}$.

	Suppose first that $a=n$. If $z=s_n$, then $b=n$ and $z=[n,n]$. Thus, we may assume that $z \neq s_n$. By Lemma \ref{lem:techlem1} and $R(z_1) = \{s_b\}$, there exists $z_1<[n-1,1,b]$ with  $z=s_nz_1$ with $l(z) = 1+l(z_1)$. Moreover, $L(z_1) \subseteq \{s_{n-1}\}$ and $R(z_1) = \{s_b\}$.

	If $b<n$, then, by the induction hypothesis,  $z_1\in \{[n-1,b], [n-1,1,b]\}$, whence $z \in \{[n,b],[n,1,b]\}$.

	If $b=n$, then by Lemma \ref{lem:techlem1}, there is $z_1<[n-1,1,n-1]$ such that $z = s_nz_1s_n$ with $l(z)=l(z_1)+2$. Moreover, $L(z_1) = \{s_{n-1}\} = R(z_1)$. By the induction hypothesis,  $z_1\in \{s_{n-1}, [n-1,1,n-1]\}$. Since $L(z) = \{s_{n}\} $, $z_1 \neq s_{n-1}$, so $z_1 = [n-1,1,n-1]$ and $z=[n,1,n]$.

	The case where $a<n$ and $b=n$ is similar.
\end{proof}

In this lemma, we compute the coefficient of $b_{[n,{n-1},n]}$ in $\pb_x$ for $x=[n,1,n]$ in type $C_n$ using the graded rank of the local intersection form, (see \ref{subsubsec:localintform}).

\begin{lemma}\label{lem:pbx_n1n}
	Assume that $W$ is of type $C_n$ and $n \geq 3$. Let $x = [n,1,n]$ and $y=s_ns_{n-1}s_n$. Then \[\pmk_{y,x} =\begin{cases}
			1, & p=2               \\
			0  & \text{otherwise.}
		\end{cases}\]
\end{lemma}
\begin{proof}
	Recall that \[\pmk_{y,x} = h_{y,x}-\ph_{y,x}-\sum_{y<z<x}\pmk_{y,z}\ph_{z,x}.\] We first compute $h_{y,x}$ and $\ph_{y,x}$ as the graded rank of the local intersection form, following \cite[Section 27.3]{EMTW}. The subexpressions of $[n,1,n]$ expressing $y$ are either of the form
	\begin{enumerate}
		\item $11e0\tilde{e}01$ for some subexpression $e$ of $[n-2,2]$ with $\tilde{e}$ the symmetric of $e$ in $[2,n-2]$,
		\item or $10e0\tilde{e}11$ for some subexpression $e$ of $[n-2,2]$ with $\tilde{e}$ the symmetric of $e$ in $[2,n-2]$.
	\end{enumerate} Let $n_e$ be the number of $0$ in $e$. In the first case, the subexpression has defect $n_e+1+n_e-1$. In the second case, it has defect $1+n_e+1 +n_e.$
	Thus, all subexpressions have non-negative defects. Therefore, the $d$-graded piece of the local intersection form is zero unless $d=0$. In that case, it is a single entry matrix of the form $(-2)^{n-2}$. Thus, $h_{y,x}=1$ and\[\ph_{y,x} =\begin{cases}
			0, & p=2               \\
			1  & \text{otherwise.}
		\end{cases}\]
	In particular, $\pmk_{y,x} = 0$ if $p \neq 2$.

	We now study the sum $\sum_{y<z<x}\pmk_{y,z}\ph_{z,x}$ and show that there is no $y<z< x$ such that $\pmk_{y,z}\neq 0$. Let $y<z\leq x$ such that $\pmk_{y,z}\neq 0$. If $\pmk_{y,z} \neq 0$, then $L(z) \subseteq L(y) = \{s_{n-1},s_n\}$ and $R(z) \subseteq R(y) = \{s_{n-1},s_n\}$  by Lemma \ref{lem:myx_descentset}. Note that since $y<z< x$, we must have $s_n \in L(z)$ and $s_n \in R(z)$. By Lemma \ref{lem:techlem1}, since $z\leq x$, there is $z_1<[n-1,1,n-1]$ with $l(z)= 2+l(z_1)$ and $L(z_1) =\{s_{n-1}\}=R(z_1)$.
	We conclude that $z_1\in \{s_{n-1},[n-1,1,n-1]\}$, by Lemma \ref{lem:subexpressionJ1}. Thus, $z \in \{s_ns_{n-1}s_n,[n,1,n]\} = \{x,y\}$ and since $z>y$, $z=x$. Therefore, the sum $\sum_{y<z<x}\pmk_{y,z}\ph_{z,x}=0$ and $\pmk_{y,x}=h_{y,x}-\ph_{y,x}= 1$ if $p=2$.
\end{proof}

\begin{theorem}[$p$-canonical basis for the subregular cell]\label{lem:pbasis_subreg}
	Let \(x\in J_1\).
	\begin{enumerate}[label=\textup{(\alph*)}]
		\item In types \(A_n\) and \(D_n\), for any prime $p$, one has
		      \[
			      \pb_x=b_x.
		      \]

		\item In type \(B_n\), if $p\neq 2$, then $\pb_x = b_x$. If $p=2$, the only change
		      inside the ordinary subregular cell is:
		      \[
			      \pb_{[1,2,1]}=b_{[1,2,1]}+b_{1}.
		      \]
		      All other elements \(x\in J_1\) satisfy \(\pb_x=b_x\).

		\item In type \(C_n\), if $p\neq 2$, then $\pb_x = b_x$. If $p=2$, the only elements
		      of \(J_1\) which change are of the form \([a,1,b]\) with
		      \(a,b\geq 2\). For these one has
		      \[
			      \pb_{[a,1,b]}
			      =
			      \sum_{k=1}^{\min(a,b)} b_{[a,k,b]}.
		      \]
		      All other elements \(x\in J_1\) satisfy \(\pb_x=b_x\).
	\end{enumerate}
\end{theorem}
\begin{proof}
	If $W$ is of type $A_n$ or $D_n$, this is a direct consequence of Lemma \ref{lem:pbx_shortestpath}.

	We can now assume that $W$ is of type $B_n$ or $C_n$. We proceed by induction on~$n$. The case $n=2$ can be found in \cite[Example 5.1]{JW1}. Assume now that $n\geq 3$ and that the result holds for $B_{n-1}$ and $C_{n-1}$. Let $x \in W$.

	Thanks to Lemma \ref{lem:pbx_shortestpath}, if $x = [a,b]$ for some $1 \leq a,b \leq n$, then \[\pb_x= b_x.\]

	By the induction hypothesis, if $x = [a,1,b]$ for $2 \leq a,b<n$, then \begin{itemize}
		\item in type $B_n$, as well as in type $C_n$ for $p\neq 2$, \[\pb_x=b_x,\]
		\item while in type $C_n$ for $p=2$,  \[\pb_x = \sum_{k=1}^{\min(a,b)}b_{[a,k,b]}.\]
	\end{itemize}
	If $x =[1,2,1]$, by the induction hypothesis,
	\begin{itemize}
		\item in type $B_n$ for $p=2$, \[\pb_{s_1s_2s_1} = b_{s_1s_2s_1}+b_1,\]
		\item while, in type $C_n$, as well as in type $B_n$ for $p\neq 2$, \[\pb_{s_1s_2s_1} = b_{s_1s_2s_1}.\]
	\end{itemize}

	We may thus assume that $x = [a,1,b]$ for some $2 \leq a,b \leq n$ with $n \in \{a,b\}$. Let $y \leq x$ such that $\pmk_{y,x} \neq 0$.

	\textbf{Case 1:} Suppose first that  $x = [a,1,n]$ for some $2 \leq a < n$. By Lemma \ref{lem:myx_descentset}, $s_n \in R(y)$. Moreover, $s_{n-1} \not \in R(y)$ since $y \leq [a,1,n]$ and $a<n$. We apply the right star operation to $x$ and $y$ with respect to $\{s_n,s_{n-1}\}$ to obtain $x^*$ and $y^*$. Note that $x^* =[a,1,n-1]$. By Lemma \ref{lem:myx_starop} and the induction hypothesis, we obtain \begin{itemize}
		\item in type $B_n$, as well as in type $C_n$ for $p\neq 2$,
		      \[\pmk_{y,x} = \pmk_{y^*,x^*} = \begin{cases}
				      1 \text{ if }y^* =x^*, \\
				      0 \text{ otherwise,}
			      \end{cases}\]
		\item while, in type $C_n$ for $p=2$,
		      \[\pmk_{y,x} = \pmk_{y^*,x^*} = \begin{cases}
				      1  \text{ if } y^*=[a,k,n-1] \text{ for some } 1 \leq k \leq a, \\
				      0 \text{ otherwise.}
			      \end{cases}\]
	\end{itemize}
	Observing that if $y^*=[a,k,n-1] $ for some $1 \leq k \leq a$, then $y = [a,k,n]$, we deduce that
	\begin{itemize}
		\item in type $B_n$, as well as in type $C_n$ for $p\neq 2$, \[\pb_x=b_x,\]
		\item while, in type $C_n$ for $p=2$, \[\pb_{[a,1,n]}=\sum_{k=1}^{\min(a,n)} b_{[a,k,n]}.\]
	\end{itemize}

	\textbf{Case 2:} Assume that $x = [n,1,b]$ for some $2 \leq b < n$. Note that $x^{-1}=[b,1,n]$ corresponds to \textbf{Case 1}. Applying Lemma \ref{lem:myx_inverse} allows us to conclude that,
	\begin{itemize}
		\item in type $B_n$, as well as in type $C_n$ for $p\neq 2$, \[\pb_{x} = b_x\]
		\item while, in type $C_n$ for $p=2$, \[\pb_x = \sum_{y \in W}\pmk_{y^{-1},x^{-1}}b_y
			      =
			      \sum_{k=1}^{b} b_{[b,k,n]^{-1}} = \sum_{k=1}^{b} b_{[n,k,b]}.\]
	\end{itemize}

	We may now finally assume that $x = [n,1,n]$.

	\textbf{Case 3a:} Assume that $x=[n,1,n]$ and that $s_{n-1} \not \in L(y)$. We apply the left star operation to $x$ and $y$ with respect to $\{s_n,s_{n-1}\}$ to obtain ${}^*x$ and ${}^*y$. Since ${}^*x =[n-1,1,n]$, Lemma \ref{lem:myx_starop} and \textbf{Case 1} imply that:
	\begin{itemize}
		\item in type $B_n$, as well as in type $C_n$ for $p\neq 2$,
		      \[\pmk_{y,x} = \pmk_{{}^*y,{}^*x} = \begin{cases}
				      1 \text{ if }{}^*y={}^*x \iff y=x, \\
				      0 \text{ otherwise,}
			      \end{cases}\]
		\item while, in type $C_n$ for $p=2$,
		      \[\pmk_{y,x} = \pmk_{{}^*y,{}^*x} = \begin{cases}
				      1  \text{ if for some  } 1 \leq k \leq n-1, \quad {}^*y=[n-1,k,n], \\
				      0 \text{ otherwise.}
			      \end{cases}\]
	\end{itemize}
	We observe that if ${}^*y=[n-1,k,n]$, then $y=[n,k,n]$ unless $k=n-1$, in which case $y = s_n$.

	\textbf{Case 3b:} Assume that $W$ is either of type $B_n$ or of type $C_n$ with $p\neq 2$, and that $x=[n,1,n]$. We show that if $y<x$ with $\pmk_{y,x} \neq 0$ then $s_{n-1} \not \in L(y)$.
	Since $\pmk_{y,x} \neq 0$, we see that $b_y$ appears as a summand in $b_{s_n}\pb_{[n-1,1,n]}$ by Remark \ref{rmk:myx_zero}. Now by \textbf{Case 1}, we know $\pb_{[n-1,1,n]}$. We compute
	\begin{align*}
		b_{s_n}\pb_{[n-1,1,n]} & = b_{s_n}b_{[n-1,1,n]}                                    \\
		                       & =b_{[n,1,n]}+ \sum_{z<[n-1,1,n]}\mu^z_{s_n,[n-1,1,n]}b_z.
	\end{align*}
	In particular, since $y<x$, we must have $y<[n-1,1,n]$. Since $s_n \in R(y)$, there is $y_1<[n-1,1,n-1]$ such that $y=y_1s_n$ with $l(y)=l(y_1)+1$. Now, since $s_n \in L(y)$, we must have $y_1 \in W_{\{s_1, \dots, s_{n-2}\}}$ and in particular, $s_{n-1} \not \in L(y)$.

	Therefore, if $W$ is of type $B_n$ or of $C_n$ with $p\neq 2$, then by \textbf{Case 3a}, \[\pb_{[n,1,n]} =b_{[n,1,n]}.\]

	\textbf{Case 3c:} Lastly, assume that $W$ is of type $C_n$, $p=2$ and $x =[n,1,n]$. Using \textbf{Case 1} we have
	\begin{align*}
		b_{s_n}\pb_{[n-1,1,n]} & = \sum_{k=1}^{n-1}b_{s_n}b_{[n-1,k,n]}.                                                    \\
		                       & =\sum_{k=1}^{n-1}b_{[n,k,n]}   +\sum_{k=1}^{n-1}\sum_{z<[n-1,k,n]}\mu^z_{s_n,[n-1,k,n]}b_z
	\end{align*}
	By the same reasoning as for \textbf{Case 3b}, we conclude that, if $\pmk_{y,x} \neq 0$, then either $y\in W_{
				\{s_1,\dots,s_{n-2}\}}s_n$ (in particular $s_{n-1}\not \in L(y)$) or $y = [n,k,n]$ for some $1 \leq k \leq n$. By \textbf{Case 3a}, we know that, if $s_{n-1} \not \in L(y)$, then $\pmk_{y,x} = 1$ if and only if $y = [n,k,n]$ for some $1 \leq k \leq n$ (with $k \neq n-1$).
	If $s_{n-1} \in L(y)$, then $y = s_ns_{n-1}s_n$ and $\pmk_{y,x}=1$ by Lemma \ref{lem:pbx_n1n}.
	Therefore, \[\pb_{[n,1,n]}
		=
		\sum_{k=1}^{n} b_{[n,k,n]}.\]

	This concludes the description of the elements $\pb_x$ for $x \in J_1$.
\end{proof}

\subsection{The $p$-cells coming from the subregular cell}

We now consider how the cell $J_1$ distributes in characteristic $p$. We start by showing that $J_1$ decomposes into a union of $p$-cells.

The first step is to verify that for $v \in J_1$, the canonical basis element $b_v$ does not appear as a coefficient of any $\pb_w$ for $w \not \in J_1$.
\begin{lemma}
	Let $w \in W \setminus \{id\}$ and $v \in J_1$. Then \[\pmk_{v,w} \neq 0 \implies w \in J_1.\]
\end{lemma}
\begin{proof}
	We first observe that for $w \in W \setminus \{id\}$, we always have $w \geq_J^0 v$. Indeed, let $s \in L(w)$. Then, $\mu_{s,sw}^w = 1$, hence $s \leq_R^0 w$. Thus, $s \leq_J^0 w$ and since $s \in J_1$, also $v \leq_J^0 w$.

	Now let $w \in W \setminus \{id\}$ and $v \in J_1$. We show that, if $\pmk_{v,w} \neq 0$, then $w \in J_1$.
	We proceed by induction on $l(w)$. If $l(w)=1$, then $\pb_w = b_w$ and $v=w \in J_1$.

	Now assume that $l(w)>1$ and $\pmk_{v,w} \neq 0$. By Lemma \ref{lem:myx_descentset},  $L(w) \subseteq L(v)$. Let $s \in S$ such that $L(v) = \{s\}$. Then $b_v$ appears as summand of $b_s\pb_{sw}$ (see Remark \ref{rmk:myx_zero}). In other words, either $v = w$ and $w \in J_1$ or there is $v' \in W$ such that $\pmk_{v',sw} \neq 0$ and $v \geq_L^0 v'$. If $v'= id$, then $w=s$ and $w \in J_1$. Thus, we may assume that $v'\neq id$, and so $v \leq_J^0 v'$. Therefore, $v \sim_J^0 v'$, whence $v' \in J_1$. By induction hypothesis, we then have $sw \in J_1$. Knowing that $L(w) = \{s\}$ and using the description of the subregular cell $J_1$ (Convention \ref{not:subregular}), we deduce that $w \in J_1$.
\end{proof}

Next we will show that $J_1$ decomposes as a union of $p$-cells.

\begin{lemma}[Union of p-cells]\label{lem:J1unionpcells}
	Let $x \in J_1$ and $y \in W \setminus \{id\}$. Then for $X \in \{L,R, J\}$, if $x \geq_X^p y$, then $y \in J_1$.
	In particular, for any $y \in W \setminus \{ id\}$, either $y \in J_1$ or $y\not \sim_J^p x$.
\end{lemma}
\begin{proof}Consider the case $X = R$, the case $X=L$ being analogous.
	Since $x \geq_R^p y$, there is $z \in W$ such that $\pmu_{z,y}^x \neq 0$. Thus, there is $v,w \in W$ such that $\pmk_{v,z}\pmk_{w,y}\mu_{v,w}^x \neq 0$ by \eqref{eq:pmuxyz}. Since $\mu_{v,w}^x \neq 0$, we have $x \geq_R^0 w$, hence $x \sim_J^0 w$. Moreover, since $\pmk_{w,y} \neq 0$, we must have $y  \in J_1$ by the previous Lemma.

	Suppose that $x \geq_J^p y$. Then there is $w \in W$ such that $x \geq_R^p w$ and $w \geq_L^p y$. Thus, $w \in J_1$ and $y \in J_1$.

	We conclude by observing that if $y \sim_J^p x$, then $x \geq_J^p y$, whence $y \in J_1$.
\end{proof}

We will now compute the $p$-cells in low rank, which will serve as a base for induction.

\begin{lemma}
	\label{lem:low-rank-input-subreg}
	Assume $p=2$ and consider Weyl groups of low rank. Then the subregular cell $J_1$ splits into $p$-cells as follows.
	\begin{enumerate}[label=\textup{(\alph*)}]
		\item If $W$ is of type $A_2$, then $J_1$ is a two-sided $p$-cell.
		\item If $W$ is of type $B_2$, then $J_1$ decomposes into two two-sided $p$-cells: $\{s_1\}$ and $J_1\setminus \{s_1\}$.
		\item If $W$ is of type $C_2$, then $J_1$ decomposes into two-sided $p$-cells: $\{s_2\}$ and $J_1\setminus \{s_2\}$.
	\end{enumerate}
\end{lemma}
\begin{proof}
	This is an explicit computation in the \(p\)-canonical basis. In type $A_2$, the $p$-canonical basis is the same as ordinary canonical basis.
	In the
	notation of Lemma~\ref{lem:pbasis_subreg}, in type \(B_2\),  the only change is
	\[
		\pb_{[1,2,1]}=b_{[1,2,1]}+b_1,
	\]
	while in type \(C_2\), the only change is
	\[
		\pb_{[2,1,2]}=b_{[2,1,2]}+b_2.
	\]
	The resulting \(p\)-cell preorder is the one stated above.
\end{proof}

%
%
The following theorem describes the subregular $p$-cells in classical types. Note that in type $A_n$, Jensen showed much more generally that $p$-cells and $0$-cells coincide, see \cite[Theorem 4.33]{J1}. We include this case for completeness and also give the proof with our methods to facilitate understanding of the other types. The notation \(a_{r,s}\) denotes an \(r\times s\) block all of whose entries are \(a\). Thus \(2_{n-1,n-1}\) is the \(n-1\times n-1\) matrix where each entry is $2$.

\begin{theorem}[The subregular \(p\)-cells]
	\label{prop:subregular-dec-pcells}
	If $p\neq 2$ or if $W$ is of type \(A_n,D_n\), the subregular $0$-cell $\J_1$ is a two-sided $p$-cell. It contains the longest element of a parabolic subgroup.

	In types \(B_n\) and \(C_n\), if $p=2$, the ordinary subregular cell \(J_1\) is the union
	of two two-sided \(p\)-cells. Order the left and right cells by
	\[
		(s_1\mid s_2,\ldots,s_n),
	\]
	Then:
	\begin{enumerate}[label=\textup{(\alph*)}]
		\item In type \(B_n\), one \(p\)-cell is strongly regular of size \(1\) and contains $s_1$,
		      and the other has shape
		      \[
			      \begin{array}{c|c}
				      1_{1,1}   & 1_{1,n-1}   \\ \hline
				      1_{n-1,1} & 2_{n-1,n-1}
			      \end{array}.
		      \]
		      It also contains the longest element of a parabolic subgroup.
		\item In type \(C_n\), one \(p\)-cell is strongly regular of size
		      \((n-1)\times(n-1)\), and the other has shape
		      \[
			      \begin{array}{c|c}
				      2_{1,1}   & 1_{1,n-1}   \\ \hline
				      1_{n-1,1} & 1_{n-1,n-1}
			      \end{array}.
		      \]
		      The strongly regular $p$-cell contains the elements of the form $[a,b]$ for $2 \leq a,b \leq n$. Both $p$-cells contain the longest element of a parabolic subgroup.
	\end{enumerate}
\end{theorem}

\begin{proof}
	We use the known relations for Weyl groups of rank $2$ and $3$ and propagate them using star operations, concluding by a counting argument.

		{\bf Type $A_n$.}
	Assume that $W$ is of type $A_n$.

	We first show that the left $0$-cell containing $1$ does not split, i.e. for all $1 \leq a \leq n$, $$[a,1] \sim_L^p s_1.$$ If $a=2$, then it is Lemma \ref{lem:low-rank-input-subreg} applied to $W_{\{s_{1},s_{2}\}}$. By transitivity of the left cell order, it is enough to show that, for any $2 < a \leq n$, we have $[a,1] \sim_L^p [a-1,1]$. By applying right star operations successively, each time using Lemma \ref{lem:starop_cell} and \cite[Theorem 3.9]{J1}, we have \[[a,1] \sim_L^p [a-1,1] \iff [a,a-1] \sim_L^p [a,a],\] and the right-hand side of the equivalence holds by Lemma \ref{lem:low-rank-input-subreg} applied to $W_{\{s_{a},s_{a-1}\}}$.

	Thanks again to Lemma \ref{lem:starop_cell}, we thus conclude that each left $0$-cell is left unchanged. Thus, the whole subregular cell $J_1$ is a two-sided $p$-cell.

	{\bf Type $D_n$.}
	Assume that $W$ is of type $D_n$ for $n \geq 4$. Restricting to $W_{\{s_2,s_3,\dots,s_n\}}$, we see that all shortest paths of the form $[a,b]$ for $2 \leq a,b \leq n$ lie in the same two-sided $p$-cell. On the other hand, restricting to  $W_{\{s_1,s_3,\dots,s_n\}}$, we see that for all $3 \leq a \leq n$ \[s_1 \sim^p_L [a,1] \sim_R^p s_a \sim_L^p [1,a].\]
	Thus, the whole subregular cell $J_1$ is a two-sided $p$-cell.

	{\bf Types $B_n$ and $C_n$.}
	Assume now that $W$ is of type $B_n$ or $C_n$. By restricting to $W_{\{s_2, \dots, s_n\}}$, we observe that all the shortest paths of the form $[a,b]$ for $2 \leq a,b \leq n$ lie in the same two-sided $p$-cell. Moreover, by \cite[Theorem 3.9]{J1}, if $x \leq_R^p y \in W_I$, then for any $w \in W^I$, $wx \leq_R^p wy$. In particular, for $2\leq a, b,c\leq n$, since $[2,b] \sim_R^p [2,c]$, we must have $[a,1,b] \sim_R^p [a,1,c]$ and $[b,1,a] = [a,1,b]^{-1} \sim_L^p [c,1,a]$. Therefore, all the elements of the form $[a,1,b]$ for $2 \leq a,b \leq n$ lie in the same two-sided $p$-cell.

	Using Lemma \ref{lem:starop_cell} again, we observe the following, for $2 \leq a,b \leq n$:
	\begin{itemize}
		\item  $[2,2] \sim_L^p [2,1,2] \iff[2,b]\sim_L^p [2,1,b]$,
		\item $ [2,2] \sim_L^p [2,1,2]  \iff [a,2]\sim_L^p [a,1,2],$
		\item $[2,1,2] \sim_J^p [2,1] \iff [a,1,2]\sim_J^p [a,1]$
		\item and $ [2,1,2] \sim_J^p [1,2] \iff [2,1,b]\sim_J^p [1,b].$
	\end{itemize}

	Thus, by Lemma \ref{lem:low-rank-input-subreg}, in type $B_n$, when $p=2$, the two-sided $p$-cell containing $s_2$ contains every element of $J_1$ except $\{s_1\}$.
	In type $C_n$, when $p=2$, the two-sided $p$-cell containing $s_2$ contains exactly the shortest paths of the form $[a,b]$ for $2 \leq a,b \leq n$. Moreover, the two-sided $p$-cell containing $[2,1,2]$ consists exactly of the remaining elements of $J_1$.
\end{proof}

\section{The submaximal cell in classical type}

Following the strategy of the previous section, we will now describe the $p$-cells in the submaximal cell.

\begin{conv}[The submaximal cell]\label{not:submaximal}
	Let \(W\) be of classical type. We write \[
		\Jtop=w_0J_1
	\]
	for the \textbf{submaximal} two-sided cell, i.e. the cell closest to \(w_0\). In general, for an element $w \in W$, we write $$\overline{w} := w_0w.$$


	The shape of the ordinary cell \(\Jtop\) is the same as the shape of
	\(J_1\). Ordered by
	\[
		(s_1\mid s_2,\dots,s_{n}),
	\]
	i.e. the exceptional endpoint first and the simply-laced tail second, it is
	\[
		\begin{array}{cc}
			2_{1,1}   & 1_{1,n-1}    \\[1mm]
			1_{n-1,1} & 2_{n-1,n-1}.
		\end{array}
	\]
\end{conv}

\begin{remark}
	The notation is chosen so that the formulae are precisely the
	\(w_0\)-translates of the notation in the subregular cell. In particular,
	we do not use any new combinatorics for the ordinary cell shape.
\end{remark}

\subsection{The $p$-Kazhdan--Lusztig elements}

We first consider how the $p$-canonical basis elements in $\Jtop$ could possibly decompose.

The first step is to reduce the possibilities of the $b_y$ appearing as coefficients in the $\pb_x$ for $x \in \Jtop$. In particular, we show that $y \in \Jtop$ unless $W$ is of type $B_n$ or $C_n$ and $x$ is of the form $w_0s_1$ or $w_0s_1s_2s_1$. To do so, we extensively use the fact that the left and right descent of $y$ are completely determined by the one of $x$. We then proceed by induction repeatedly using the star operation.

\begin{lemma}\label{lem:possiblemyx_submin} Let $y \in W$ such that $\pmk_{y,x} \neq 0$ for some $x \in \Jtop$. If $y \not\in \Jtop$, then $W$ is of type $B_n$ or $C_n$ and $y = \overline{v}$ where $v$ is of the form $s_1s_2s_1[3,1]\dots[k,1]$ for $3\leq k \leq n$.
\end{lemma}
\begin{proof} Let $y \in W$ and $x \in \Jtop$ such that $\pmk_{y,x} \neq 0$. Since $y \leq x$ in the Bruhat order, we must have $y \neq w_0$.

	By definition of $\Jtop$ and Lemma \ref{lem:myx_descentset}, we have $L(x)=L(y)$ and $R(x)=R(y)$ with both sets of size $n-1$. Let $1 \leq a, b\leq n$ such that $s_a \not \in L(y)$ and $s_b\not\in R(y)$. In other words, $L(\overline{y})= \{w_0s_aw_0\}$ and $R(\overline{y})= \{s_b\}$. There exists $z \in W$ such that $\overline{y} = zs_b$ and $l(\overline{y}) = l(z)+1$. Moreover, if $s\in R(z)$, then $s \in R(\overline{y})$ unless $ss_b \neq s_bs$, whence $R(z) \subseteq \{s_{b+1},s_{b-1}\}$ (if $1 \leq b+1, b-1 \leq n$).

	We will repeatedly use the following fact, a consequence of Lemma \ref{lem:starop_cell}: \textit{if we apply a (left/right) star operation to $x$ then we obtain a new element $x^*$ which also belongs to $\Jtop$.}

	\textbf{Case 0} If $R(z) = \emptyset$, then $y = \overline{s_b} \in \Jtop$.

	\textbf{Case 1} Assume that $b+1 \in R(z)$. If $W$ is of type $B_n$ or $C_n$, assume furthermore that $b \geq 2$.
	We use induction on the length of $\overline{y}$ to show that $y \in \Jtop$.

	If $l(\overline{y}) = 2$, then $\overline{y} = s_{b+1}s_{b} \in J_1$ and $y \in \Jtop$.

	Now assume that $l(\overline{y}) =l>2$. Assume the following induction hypothesis: let $y' \in W$ with $l(\overline{y'}) <l$ such that $\pmk_{y',x'} \neq 0$ for some $x'\in \Jtop$. Let $1 \leq c \leq n$ ($c\geq 2$ if $W$ is of type $B_n$ or $C_n$) with $R(x') = S\backslash \{s_c\}$. If $y' = z's_c$ for some $z' \in W$ with $s_{c+1} \in R(z')$ and $l(y') = l(z')+1$, then $y' \in \Jtop$.

	We apply the right star operation to $y$ and $x$ with respect to $\{s_b,s_{b+1}\}$ (or $\{s_1,s_3\}$ if $b=1$ and $W$ is of type $D$) to obtain $y^*$ and $x^* \in \Jtop$. By definition of the star operation and since $x^* \in \Jtop$, we have $R(x^*) = S \backslash \{s_{b+1}\}$. By Lemma \ref{lem:myx_starop}, $\pmk_{y^*,x^*} = \pmk_{y,x} \neq 0$. Therefore $R(y^*) = S \backslash \{s_{b+1}\}$ and $R(\overline{y^*}) = \{s_{b+1}\}$. Hence, there is $z'\in W$ such that $\overline{y^*} = z's_{b+1}$, $l(\overline{y^*}) = l(z')+1$ and $R(z') \subseteq \{s_{b+2},s_b\}$. Since $l>2$, $R(z') \neq \emptyset$. On the other hand, $s_{b} \not \in R(z')$ since otherwise $s_{b+1} \in R(y)$. Thus, $s_{b+2} \in R(z')$. By the induction hypothesis, $y^*=\overline{z} \in \Jtop$ whence $y \in \Jtop$ by Lemma \ref{lem:starop_cell}.

	\textbf{Case 2} Suppose now that $s_{b-1} \in R(z)$. In type $B_n$ and $C_n$, assume furthermore that $b\geq 2$. We check that $y \in \Jtop$.

	\textbf{Step 1} We show that either $y \in \Jtop$ or $W$ is of type $B_n$ or $C_n$ and $\overline{y}=w[1,b]$ for $w \in W$ with $l(\overline{y})=l([1,b])+l(w)$.

	If $W$ is of type $B_n$ or $C_n$ and $b=2$, then since $s_1 \in R(z)$, $\overline{y}=w[1,2]$ for $w \in W$ with $l(\overline{y})=2+l(w)$. Thus, if $W$ is of type $B_n$ or $C_n$, we may assume that $b \geq 3$.

	Let us now proceed with an inductive proof. If $l(\overline{y}) = 2$, then $\overline{y} = s_{b-1}s_{b}$ and $y \in \Jtop$.

	Now assume that $l(\overline{y}) =l >2$. We apply the right star operation to $y$ and $x$ with respect to $\{s_b,s_{b-1}\}$ to get $y^*$ and $x^* \in \Jtop$. Note that $\overline{y}^*  = z$ and $ y^*= w_0z$ by Lemma \ref{lem:starop_w0}. Thus, $l(\overline{y^*}) =l-1<l$. By Lemma \ref{lem:myx_starop}, $\pmk_{y^*,x^*} = \pmk_{y,x} \neq 0$. In particular, $R(\overline{y}^*) = \{s_{b-1}\}$ and there is $z'\in W$ such that $\overline{y^*} = z's_{b-1}$ and $l(\overline{y^*}) = l(z')+1 = l-1$. As before, $R(z')\subseteq \{s_b,s_{b-2}\}$ and $R(z') \neq \emptyset$ since $l(z') = l(\overline{y})-2 \geq1$. If $s_{b} \in R(z')$, then $s_{b-1} \in R(\overline{y})$, a contradiction. Thus, $R(z') = \{s_{b-2}\}$. If $b-1 =2$, then $\overline{y} = ws_1s_2s_3$ for some $w \in W$ with $l(\overline{y})= 3 +l(w)$. If $b-1>2$, then by induction hypothesis, either $y^* \in \Jtop$, whence $y \in \Jtop$ or $\overline{y^*}=w[1,b-1]$ for $w \in W$ with $l(\overline{y^*})=l([b-1,1])+l(w)$, whence $\overline{y} = w[1,b]$.

	\textbf{Step 2} Assume that $W$ is of type $B_n$ or $C_n$ and that $\overline{y} = w[1,b]$ for some $w \in W$ with $l(\overline{y})=l(w)+l([1,b])$ for some $b >1$. We show that $y \in \Jtop$.

	If $w =id$, then $\overline{y} = [1,b]$ and $y\in \Jtop$. So we may assume that $w \neq id$. We will prove that $\overline{y} = [a,1,b]$ for some $1 \leq a \leq n$.

	By repeatedly applying star operations as in \textbf{Step 1}, we deduce that there exists $x'\in \Jtop$ such that $\pmk_{\overline{ws_1s_2},x'} \neq 0$. Let $y'=\overline{ws_1s_2}$. If $s_3 \in R(w)$, then by \textbf{Case 1}, $y'\in \Jtop$, whence $w$ is of the form $[a,2]$ for some $2 \leq a \leq n$, a contradiction. Thus, $s_3 \not \in R(w)$.

	We will show that for any $l(\overline{y'}) -2 \geq j \geq 1$, there is $w_{j}\in W$ such that $\overline{y'} = w_{j}[j,1,2]$ (where $[1,1,2]$ denotes $[1,2]$) with $l(\overline{y'})=j+l(w_{j})+1$ and $R(w_j) \subseteq \{s_{j+1}\}$.

	If $j=1$, we set $w_1 = w$. We need to check that $R(w)= \{s_2\}$. By assumption, $R(w) \neq \emptyset$. Since $\pmk_{y',x'} \neq 0$ and $x'\in \Jtop$, we know that $|R(y')| = n-1$ and $R(\overline{y'}) = \{s_2\}$. Thus $R(w) \subseteq \{s_3,s_2\}$, and since we have shown that $s_3 \not \in R(w)$, we conclude that $R(w) = \{s_2\}$.

	Therefore, there exists $w_2 \in W$ such that $w =w_2s_2$ and $l(w)=l(w_2)+1$. Since $R(\overline{y'}) = \{s_2\}$, we must have $R(w_2) \subseteq \{s_3\}$.
	Assume now that there is $w_k\in W$ such that $\overline{y'} = w_k[k,1,2]$ and $R(w_{k}) \subseteq \{s_{k+1}\}$ for $2\leq k<l(\overline{y'}) -2$. Since $k < l(\overline{y'}) -2$, we have $w_k \neq id$ and $R(w_k) = \{s_{k+1}\}$. Thus, there is $w_{k+1}\in W$ such that $w_{k} = w_{k+1}s_{k+1}$ and $l(w_{k})=l(w_{k+1})+1$. Since $R(w_{k}) = \{s_{k+1}\}$, we see that $R(w_{k+1})\subseteq \{s_{k},s_{k+2}\}.$ If $R(w_{k+1}) = \emptyset$, then $\overline{y'} = [k+1,1,2]$. If $s_k \in R(w_{k+1})$, then we can write $\overline{y'} = w's_{k}s_{k+1}s_{k}[k-1,1,2] = w's_{k+1}[k,1,2]s_{k+1}$ for some $w'\in W$ with $l(\overline{y'}) =l(w')+k+3$, contradicting $R(\overline{y'}) = \{s_2\}$. Thus $R(w_{k+1})\subseteq \{s_{k+2}\}.$
	We have shown that $\overline{ws_1s_2} = [a,1,2] \in \Jtop$. In particular $w=[a,2]$, whence $\overline{y} =[a,1,b] \in J_1$ and $y \in \Jtop$.



	\textbf{Case 3:} Assume that $W$ is of type $B_n$ or $C_n$ and that $b=1$. In this case, $L(\overline{y})= \{w_0s_aw_0\} =\{s_a\}$. If $a\neq 1$, a similar analysis on the left implies that $y \in \Jtop$. Thus, we may now assume that $b=a=1$.

	Observe first that if $\overline{y} = s_1s_2s_1[3,1] \dots [k,1]$ then $L(\overline{y})=R(\overline{y}) = \{s_1\}$. We show by a counting argument that there is no other $w \in W$ which satisfies $L(w)=R(w)=\{s_1\}$. We set
	\[X := \{w\in W \mid L(w)=R(w)\subseteq \{s_1\}\}.\]
	The set $X$ is in bijection with the set of double cosets $W_I\backslash W /W_I$ for $I = S\backslash \{s_1\}$, see \cite[Exercise 2.15]{BB1}. The number of these double cosets is equal to $\langle \mathrm{Ind}_{W_I}^W 1_{W_I}, \mathrm{Ind}_{W_I}^W 1_{W_I} \rangle$ \cite[Exercise 7.77a]{St}. Since $ \mathrm{Ind}_{W_I}^W 1_{W_I}$ decomposes into a sum of $n+1$ irreducible representations, each appearing with coefficient~$1$ (see \cite[Proposition 6.15]{GP}), we conclude that $|X| = n+1$. Note that the identity belongs to $X$ and is the only element with empty left or right descent set. Thus, there are exactly $n$ elements with left and right descent sets equal to~$\{s_1\}$, the ones that we have enumerated above.
\end{proof}

We are now ready to study the decomposition of the $\pb_x$ for $x \in \Jtop$.

\begin{theorem}[$p$-canonical basis for the submaximal cell]\label{lem:pbasis_submin}
	Let $x \in \Jtop$.
	\begin{enumerate}
		\item In types $A_n$ and $D_n$, we have $b_x=\pb_x$ for any prime $p$.
		\item In type $B_n$, if $p\neq 2$, then $\pb_x = b_x$. If $p=2$, then $\pb_x= b_x$ unless $x = w_0[i,j]$ for $i,j \geq 2$. In this case,
		      \[\pb_{w_0[i,j]} = b_{w_0[i,j]} + b_{w_0[i,1,j]}.\]
		\item In type $C_n$, if $p\neq 2$, then $\pb_x = b_x$. If $p=2$, then $\pb_x= b_x$ unless $x = w_0s_1$ or $x=w_0s_1s_2s_1$. If $x = w_0s_1$, then \[\pb_{w_0s_1} = b_{w_0s_1}+b_{w_0s_1s_2s_1}.\]
	\end{enumerate}
\end{theorem}
\begin{proof}Let $x\in \Jtop$ and $y\in W$ such that $\pmk_{y,x} \neq 0$.

	Assume that $W$ is of type $A_n$ or $D_n$. By Lemma \ref{lem:myx_descentset}, if $y\in \Jtop$ then $y$ and $x$ must belong to the same right and left $0$-cells. Thus, in types $A_n$ and $D_n$, Lemma \ref{lem:possiblemyx_submin} and the fact that $\Jtop$ is strongly regular imply that $\pb_x=b_x$ for each $x\in J$.

	Assume now that $W$ is of type $B_n$ or $C_n$.

	\textbf{Case 1.} Assume that $s_1 \in L(x)$. By Lemma \ref{lem:myx_descentset}, we must have $s_1 \in L(y)$. By Lemma \ref{lem:possiblemyx_submin}, we have $y \in \Jtop$, whence $y$ and $x$ must belong to the same right and left $0$-cells. Thus, if $x = w_0[a,1,b]$ for some $2 \leq a,b \leq n$, we have $\pb_x =b_x.$
	Assume now that there exist $a,b \geq 2$ such that $x = w_0[a,b]$ and $y = w_0[a,1,b]$. We show inductively on $|a-b|$ that $\pmk_{y,x} = 0$ in type $C_n$ or if $p \neq 2$ and $\pmk_{y,x} =1$ in type $B_n$ for $p=2$.

	\textbf{Base case.} Suppose first that $a=b$. We proceed by induction on $a$ and on $n$. If $a=b=2$ and $n=2$, then $x = s_1s_2s_1$ and $y=s_1$ and this has been computed in \cite[Section 5.4]{JW1}. If $n \geq 3$, then $x \in [n,1,n]W_{\{s_1,\dots,s_{n-1}\}}$ and $y \in [n,1,n]W_{\{s_1,\dots,s_{n-1}\}}$. Thanks to Lemma \ref{lem:myx_parabolic} and the induction hypothesis on $n$, we conclude that $\pb_x=b_x$ in type $C_n$ and $\pb_x = b_x+b_y$ in type $B_n$.

	Suppose now that $a=b>2$ and consider $(\starop{a-1}{a}{x})^{*_{a,a-1}} = w_0s_{a-1}$. By the induction hypothesis and Lemma \ref{lem:myx_starop}, $\pmk_{y,x} = 0$ in type $C_n$ or if $p \neq 2$ and $\pmk_{y,x} =1$ in type $B_n$ for $p=2$.

	\textbf{Induction step.} 	Assume now that $|a-b|\geq 1$. If $a>b$, then we have $\starop{a}{a-1}{x} = \starop{a}{a-1}{[a,b]}w_0=[a-1,b]w_0$ and $\starop{a}{a-1}{y}= [a-1,1,b]w_0$ by Lemma \ref{lem:starop_w0}. The induction hypothesis and Lemma \ref{lem:myx_starop} imply the result. If $a<b$, then $\starop{a}{a+1}{x} = \starop{a}{a+1}{[a,b]}w_0=[a+1,b]w_0$ and $\starop{a}{a+1}{y}= [a+1,1,b]w_0$. We conclude thanks to the induction hypothesis and Lemma \ref{lem:myx_starop}.

	\textbf{Case 2.} We assume that $s_1 \not \in L(x)$ and $s_1 \in R(x)$. By Lemma \ref{lem:myx_descentset}, we must have $s_1 \in R(y)$. Lemma \ref{lem:possiblemyx_submin} shows $y \in \Jtop$. Hence $y$ and $x$ must belong to the same right and left $0$-cells and $y=x$.

	\textbf{Case 3.} Assume that $x = w_0s_1$ and $y \in \Jtop$. We show inductively on the rank of $W$ that if $W$ is of type $B_n$ or if $p \neq 2$, then $\pmk_{y,x}=0$ and if $W$ is of type $C_n$ and $p=2$, then $\pmk_{y,x}=1$. If $n=2$, then $x = s_2s_1s_2$ and $y=s_2$ and this has been computed in \cite[Section 5.4]{JW1}. If $n \geq 3$, then $x,y \in [n,1,n]W_{\{s_1,\dots,s_{n-1}\}}$. By Lemma \ref{lem:myx_parabolic} and the induction hypothesis, we conclude $\pmk_{y,x} =0 $ in type $B_n$ or if $p \neq 2$ and $\pmk_{y,x} =1$ in type $C_n$ and $p=2$.

	\textbf{Case 4.} Assume that $x = w_0s_1$ and $y = w_0v$ where $v$ is of the form $s_1s_2s_1[3,1]\dots[k,1]$ for $2\leq k \leq n$. If $\pmk_{y,x} \neq 0$, then $b_y$ appears as a summand in $\pb_{xs_2}b_{s_2}$ by Remark \ref{rmk:myx_zero}. Now $xs_2 \in \Jtop$ with $R(xs_2) = S\backslash \{s_2\}$, thus $\pb_{xs_2} =b_{xs_2}$ by \textbf{Case 2}. In particular, it means that if $\pmk_{y,x} \neq 0$, then $y \geq_R xs_2$, whence $y \in \Jtop$. Thus $y = w_0s_1s_2s_1$ as in \textbf{Case 3}.

	\textbf{Case 5.} Assume that $x = w_0s_1s_2s_1$ and $W$ is of type $B_n$. Let $x'= w_0s_1s_2$. Observe that $x$ and $x'$ lie in the same right cell. Thus, there is $w \in W$ such that $\mu_{x',w}^x \neq 0$. Hence, there is $v \in W$ such that $\pmu_{x',w}^v \neq 0$ and $\pmk_{x,v} \neq 0$. We then consider the decomposition of $\pb_{x'}\pb_{w}$ into $b_z$ for $z \in W$. Since $\pb_{x'}=b_{x'}$ by \textbf{Case 2}, we must have $z \geq_{R} x'$, whence, either $z = w_0$ or $z \sim_R x'$. In particular, $v \sim_R x' \sim_R x$ ($v \neq w_0$ since $\pb_{w_0} = b_{w_0}$). Since $R(v) \subseteq R(x)$ as $\pmk_{x,v} \neq 0$, we must have $v \sim_L x$ as well, so $v \in \{w_0s_1,x\}$. By \textbf{Case 3},  we cannot have $v=w_0s_1$ and $\pmk_{x,v}\neq 0$, thus we conclude that $v=x$. Moreover, if $\pmk_{y,x} \neq 0$, then $y \sim_R x$. Besides, $R(y)=R(x)$ and $y\leq x$ in the Bruhat order, thus $y=x$ and $\pb_x = b_x$.
\end{proof}

\begin{remark}\label{rem:w0121_Cn}
	We have not described $\pb_{w_0s_1s_2s_1}$ in type $C_n$ for $p=2$ and $n \geq 3$. If $n=3$, we have \[\pb_{w_0s_1s_2s_1} = b_{w_0s_1s_2s_1}+ (v+v^{-1})b_{w_0u},\]
	where $u = s_1s_2s_1[3,1]$ and $w_0u=s_2s_3s_2$ (\cite[Section 5.4]{JW1}). Thus, for any $n >3$, for $y = w_0u\in [n,1,n]\dots[4,1,4]
		W_{\{s_1,s_2,s_3\}}$, we have $\pmk_{y,x}= v+v^{-1}$ by Lemma \ref{lem:myx_parabolic}.

	More generally, we conjecture that $\pmk_{w_0u,x} \neq 0$ for each $u$ of the form $s_1s_2s_1[3,1]\dots[k,1]$ for $3\leq k \leq n$. This is satisfied for $n=4$. 
\end{remark}

\subsection{The $p$-cells coming from the submaximal cell}

We are now ready to describe how the elements $\pb_x$ for $x \in \Jtop$ distribute over two-sided $p$-cells. We first choose a specific element $x$ such that $\pb_x = b_x$ and show that its two-sided $p$-cell is contained in $\Jtop$ and is submaximal (Lemma \ref{lem:tech2_submin}).

\begin{lemma}\label{lem:mxy_xclosew0}
	Let $x \in \Jtop$ such that $l(x)= l(w_0)-1$. Then  $\pmk_{x,w} = 0$ for all $w \neq x$.
\end{lemma}
\begin{proof}
	If $\pmk_{x,w} \neq 0$, we must have $x < w$ in the Bruhat order or $x = w$. In the first case, $l(w)> l(w_0)-1$ and so $w = w_0$ which contradicts Remark \ref{rmk:pbw0}. Thus $w=x$.
\end{proof}

\begin{lemma}\label{lem:biggerthanunchanged}
	Let $x \in W$ such that $\pb_x =b_x$. Let $y \in W$. For $X \in \{L,R,J\}$, if $y \geq_X^p x$, then for all $z \in W$ such that $\pmk_{z,y} \neq 0$, $z\geq_X x$. In particular, $y \geq_X x$.
\end{lemma}
\begin{proof}
	We consider the case $X=L$, the case $X=R$ being similar. Since $y \geq_L^p x$, there is $w\in W$ such that $\pmu_{w,x}^y \neq 0$. Let $z \in W$ such that $\pmk_{z,y} \neq 0$. Then, $\pmk_{z,y}\pmu_{w,x}^y \neq 0$ and by Equation \ref{eq:pmuxyz}, there is $w',x' \in W$ such that $\pmk_{w',w}\pmk_{x',x}\mu^{z}_{w',x'}\neq 0$. Since $\pb_x=b_x$, we must have $x'=x$, whence $z \geq_L x$.

	Suppose that $y \geq_J^p x$, then there is $y_1 \in W$ such that $y \geq_L^p y_1$ and $y_1 \geq_R^p x$. Let $z \in W$ such that $\pmk_{z,y} \neq 0$. Since $y \geq_L^p y_1$, there is $w \in W$ such that $\pmu_{w,y_1}^y \neq 0$. Thus, there exist $w',y_1' \in W$ such that $\pmk_{w',w}\pmk_{y_1',y_1}\mu^{z}_{w',y_1'}\neq 0$. Hence, $z \geq_L y_1'$ and, by the previous case, $y_1' \geq_R x$. Thus,  $z \geq_J x$.
\end{proof}

The following lemma describes a particular $x \in \Jtop$ such that for any $y \in \Jtop$, we can find $\tilde{y} \in \Jtop$ such that $\tilde{y} \sim_J^p x$ and $b_y$ is a coefficient of $\pb_{\tilde{y}}$.
\begin{lemma}\label{lem:tech1_submin}
	Let $x\in W$ such that $\pb_x =b_x$ and $\pmk_{x,w} = 0$ for all $w \neq x$. Let $y \in W$ such that $y \sim_J x$. Then there is $\tilde{y} \in W$ such that
	\begin{enumerate}
		\item $\pmk_{y,\tilde{y}} \neq 0$,
		\item $\tilde{y} \geq_J x$,
		\item and $\tilde{y} \sim_J^p x$.
	\end{enumerate}
	If $x \in \Jtop$, then $\tilde{y} \sim_J x$.
\end{lemma}
\begin{proof}
	Since $x \leq_J y$, $b_y$ is a summand of $\pb_z\pb_x\pb_{z'}$ for some $z,z'\in W$. In particular, there is $\tilde{y} \in W$ such that $\pmk_{y,\tilde{y}} \neq 0$ and $\tilde{y} \geq_J^p x$. By Lemma \ref{lem:biggerthanunchanged}, we have $\tilde{y} \geq_J x$. On the other hand, since $y \leq_J x$, $b_x$ is a summand of $b_zb_{y}b_{z'}$ for some $z,z'\in W$, whence $b_x$ is a summand of $\pb_z\pb_{\tilde{y}}\pb_{z'}$. By assumption, we must have $x \geq_J^p \tilde{y}$.

	Lastly, since $y \neq w_0$, we must have $\tilde{y} \sim_J x$ if $x \in \Jtop$.
\end{proof}

\begin{lemma}\label{lem:tech2_submin}
	Let $x\in \Jtop$ such that $\pb_x =b_x$. Let $\pJ$ be the two-sided $p$-cell containing $x$. Then $\pJ \subseteq \Jtop$. If, moreover, $\pmk_{x,z} = 0$ for all $z \neq x$, then for any $w \in W$, if $w \not \leq_J^p  \pJ$, then $w = w_0$.
\end{lemma}
\begin{proof}
	Let $y \geq_J^p x$ and $y \neq w_0$. Lemma \ref{lem:biggerthanunchanged} implies $y \geq_J \Jtop$. In particular, if $y \in \pJ$, then $y \in \Jtop$. If $y =w_0$, then $x \not \geq_J y$ and thus Lemma \ref{lem:biggerthanunchanged} shows $x \not \geq_J^p y$ and $y \not \in \pJ$. Thus, $\pJ \subseteq \Jtop$.

	Finally, let $w \in W$. By definition of $\Jtop$, we have $w \leq_J x$ or $w=w_0$. If $w \neq w_0$, then $b_x$ is a summand of $b_zb_{w}b_{z'}$ for some $z,z'\in W$, whence $b_x$ is a summand of $\pb_z\pb_{w}\pb_{z'}$. By assumption on $x$, we must have $x \geq_J^p w$. Thus, if $w \not \leq_J^p x$, then $w=w_0$.
\end{proof}

We will now describe the two-sided $p$-cells containing the elements of $\Jtop$. Again, the submaximal cell in type $A_n$ (and in fact any $0$-cell) was already treated in \cite[Theorem 4.33]{J1}. We include our proof for convenience of the reader. We first choose an element $x \in \Jtop$ as in Lemma \ref{lem:tech2_submin} and describe its two-sided $p$-cell $J$. In type $B_n$ (if $p=2$), we then choose another element $x'$ with $\pb_{x'} = b_{x'}$ and consider its two-sided cell $J'$. We show that those two two-sided cells are disjoint using a generalisation of the star operation and that $J \sqcup J' = \Jtop$.

\begin{theorem}[The submaximal $p$-cells]\label{thm:submax-dec-pcells}
	If $p\neq 2$ or if $W$ is of type $A_n$ or $D_n$, the submaximal $0$-cell $\Jtop$ is a two-sided $p$-cell. It contains the longest element of a parabolic subgroup.

	Assume that $p =2$.
	\begin{enumerate}
		\item In type $C_n$, the set $\Jtop \backslash \{w_0s_1s_2s_1\}$ forms a single two-sided $p$-cell. It contains the longest element of a parabolic subgroup.
		\item In type $B_n$, the submaximal two-sided $0$-cell \(\Jtop=w_0J_1\) is the union of
		      two two-sided \(p\)-cells.  Order the left and right cells by
		      \[
			      (s_1\mid s_2,\ldots,s_n),
		      \]
		      Then one \(p\)-cell is strongly regular of size
		      \((n-1)\times(n-1)\) containing the elements of the form $w_0[a,1,b]$ for $2 \leq a,b \leq n$. It contains the longest element of a parabolic subgroup. The other two-sided $p$-cell has shape
		      \[
			      \begin{array}{c|c}
				      2_{1,1}   & 1_{1,n-1}   \\ \hline
				      1_{n-1,1} & 1_{n-1,n-1}
			      \end{array}.
		      \] It does not contain the longest element of a parabolic subgroup.
	\end{enumerate}

\end{theorem}
\begin{proof}
	{\bf Types $A_n$ and $D_n$.}
	Assume that $W$ is of type $A_n$ or $D_n$. Let $x = w_0s_1 \in \Jtop$. Theorem \ref{lem:pbasis_submin} shows $\pb_x =b_x$ and Lemma \ref{lem:mxy_xclosew0}  implies $\pmk_{x,w} = 0$ for all $w \neq x$. Let $\pJ$ be the two-sided $p$-cell containing $x$. Lemma \ref{lem:tech2_submin} yields $\pJ \subseteq \Jtop$.
	We are left to show that $\Jtop \subseteq \pJ$. Let $y \in \Jtop$. By Lemma \ref{lem:tech1_submin}, there is $\tilde{y} \in \Jtop$ such that $\pmk_{y,\tilde{y}} \neq 0$ and $\tilde{y} \in \pJ$. By Theorem \ref{lem:pbasis_submin}, we must have $\tilde{y}=y$, and thus $y \in \pJ$.

		{\bf Type $C_n$.}
	Assume now that $W$ is of type $C_n$. Let $x = w_0s_2 \in \Jtop$. Theorem \ref{lem:pbasis_submin} guarantees $\pb_x =b_x$ and Lemma \ref{lem:mxy_xclosew0} shows $\pmk_{x,w} = 0$ for all $w \neq x$. Let $\pJ$ be the two-sided $p$-cell containing $x$. Lemma \ref{lem:tech2_submin} implies $\pJ \subseteq \Jtop$.

	Let $y \in \Jtop$ such that $y \not \in \{w_0s_1s_2s_1\}$. By Theorem \ref{lem:pbasis_submin},  if $\tilde{y} \in \Jtop$ such that $\pmk_{y,\tilde{y}} \neq 0$, then $\tilde{y}=y$. Thus by Lemma \ref{lem:tech1_submin}, we must have $y \in \pJ$.

	If $y=w_0s_1s_2s_1$ and $p\neq 2$, the same argument applies and $y \in \pJ$.
	If $y=w_0s_1s_2s_1$ and $p=2$, then by Remark \ref{rem:w0121_Cn} there is $z \in W$ with $\pmk_{z,y} \neq 0$ and $z <_J x$. Thus, Lemma \ref{lem:biggerthanunchanged} shows $y \not \geq_J^p x$ and $y \not \in \pJ.$

	{\bf Type $B_n$.}
	Finally, assume that $W$ is of type $B_n$. Let $x = w_0s_1 \in \Jtop$. Theorem \ref{lem:pbasis_submin} yields $\pb_x =b_x$ and Lemma \ref{lem:mxy_xclosew0} shows $\pmk_{x,w} = 0$ for all $w \neq x$. Let $\pJ$ be the two-sided $p$-cell containing $x$. Lemma \ref{lem:tech2_submin} implies $\pJ \subseteq \Jtop$. If $p\neq 2$, then for any $y \in \Jtop$, there is no $\tilde{y} \neq y \in \Jtop$ such that $\pmk_{y,\tilde{y}} \neq 0$, by Theorem \ref{lem:pbasis_submin}. Thus, Lemma \ref{lem:tech1_submin} guarantees $y \in \pJ$.

	From now on, let $p=2$. Let $y \in \Jtop$ such that $y \neq w_0[i,1,j]$ for some $i,j \geq 2$. By Theorem \ref{lem:pbasis_submin}, there is no $\tilde{y} \neq y \in \Jtop$ such that $\pmk_{y,\tilde{y}} \neq 0$. Thus,  Lemma \ref{lem:tech1_submin} yields $y \in \pJ$.

	Now let $x' = w_0[2,1,2] \in \Jtop$. Let $\pJ'$ be the two-sided $p$-cell containing $x'$. Applying Lemma \ref{lem:starop_cell} repeatedly, we deduce $x'\sim_L^p w_0[a,1,2]$ for any $2 \leq a \leq n$. Moreover, $w_0[2,1,b] \sim_R^p w_0[a,1,b]$ for any $2 \leq a,b \leq n$. Thus, $\pJ'$ contains all the elements of $\Jtop$ of the form $w_0[a,1,b]$ for any $2 \leq a,b \leq n$.
	On the other hand, Lemma \ref{lem:tech2_submin} guarantees $\pJ'\subseteq \Jtop$.

	Lastly, we check that $\pJ' \neq \pJ$. It suffices to show that $x'\not \in \pJ$. Observe that $b_{w_0s_2}$ is a summand of $b_{x'}b_{s_2s_1}$. Thus, $\pb_{w_0s_2}$ is a summand of $\pb_{x'}\pb_{s_2s_1}$ and $x' \leq_R^p w_0s_2$. If $x' \in \pJ$, then there is $y \in W$ such that $x'\geq_L^p y \geq_R^p w_0s_2$. Thus, $L(w_0s_2)\subseteq L(y)$ and $R(y) \subseteq R(x')$. By Lemma \ref{lem:tech2_submin}, we must have $y\in \pJ$, hence $y \in \Jtop$. Thus, either $y=x'$ or $y=w_0s_2$. If $y =w_0s_2$, then $x'=(x')^{-1} \geq_R^p y^{-1}=y$. Thus, if $x' \in \pJ$, then $x'\geq_R^p w_0s_2$. So we need to show that $x' \not \geq_R^p w_0s_2$.
	To do so, we will check that for all $y \sim_R w_0s_2$ not of the form $w_0[i,1,j]$ for some $i,j \geq 2$, for all $z \sim_R w_0s_2$ of the form $w_0[i,1,j]$ for some $i,j \geq 2$ and for all $w \in W$, we have $\pmu_{y,w}^{z} =0$.
	Equivalently, we verify that for each $y$ as before, $z = w_0[2,1,c]$ and $w \in W$, the coefficient of $b_{z}$ and the coefficient of $b_{w_0[2,c]}$ are equal in $\pb_y\pb_w$. We write $\pmu(y,w,z)$ for the coefficient of $b_{z}$ in $\pb_y\pb_w$. We proceed by induction on $l(w)$.

	\textbf{Base case} Let $s \in S$. Assume first that $s\in R(y)$. If $y = w_0[2,b]$ for some $2 \leq b \leq n$, then $b_yb_s = (v+v^{-1})b_y$ and $b_{w_0[2,1,b]}b_s = (v+v^{-1})b_{w_0[2,1,b]}$. If $y = w_0s_2s_1$, then $b_yb_s = (v+v^{-1})b_y$. In both cases, the hypothesis is verified.

	Assume that $s \not \in R(y)$, We use a variant of the star operation which applies to $s_1$ and $s_2$ as defined in \cite[Definition 4.1]{J1}. It exchanges $w_0[2,1,a]$ and $w_0[2,a]$ for any $2 \leq a \leq n$ and fixes $w_0[2,1]$. Moreover, thanks to \cite[Corollary 4.10]{J1}, for $w,v \in \{w_0[2,1,a],w_0[2,a]\}$, we have \[\mu_{w,s_a}^v = \mu_{w^*,s_a}^{v^*}. \]

	Let $z = w_0[2,1,c]$ for some $2 \leq c \leq n$. If $y = w_0[2,b]$ for some $2 \leq b \leq n$, then  \[\pmu(y,s_b,z)=\mu_{w_0[2,b],s_b}^{w_0[2,1,c]}+\mu_{w_0[2,1,b],s_b}^{w_0[2,1,c]} = \mu_{w_0[2,1,b],s_b}^{w_0[2,c]}+\mu_{w_0[2,b],s_b}^{w_0[2,c]} = \pmu(y,s_b,w_0[2,c]).\]

	Assume finally that $y=w_0s_{2}s_{1}$. Then,  \[\pmu(y,s_1,z) = \mu_{w_0[2,1],s_1}^{w_0[2,1,c]}= \mu_{w_0[2,1],s_1}^{w_0[2,c]}=\pmu(y,s_1,w_0[2,c]).\] Therefore the hypothesis is verified.

	\textbf{Induction step} Let $w \in W$ such that $l(w) \geq 2$. Let $s \in R(w)$ and $w_1 =ws$. Then $\pb_{w} = \pb_{w_1}b_s - \sum_{v<w_1}\pmu_{w_1,s}^v\pb_v$. Let $y\sim_R w_0s_2$ not of the form $w_0[i,1,j]$ for some $i,j \geq 2$. Let $z = w_0[2,1,c]$ for some $2 \leq c \leq n$. We then have
	\[\pmu(y,w,z) = \sum_{z' \in W}\pmu_{y,w_1}^{z'}\pmu(z',s,z) - \sum_{v <w_1}\pmu_{w_1,s}^v\pmu(y,v,z)\]
	By the induction hypothesis, if $\pmu_{y,w_1}^{z'} \neq 0$, then  $z'\sim_R w_0s_2$ and $z'$ is not of the form $w_0[i,1,j]$ for some $i,j \geq 2$. Thus applying the induction hypothesis to each $v$ and by the initial step, we obtain
	\begin{align*}
		\pmu(y,w,z) & = \sum_{z' \in W}\pmu_{y,w_1}^{z'}\pmu(z',s,w_0[2,c]) - \sum_{v <w_1}\pmu_{w_1,s}^v\pmu(y,v,w_0[2,c]) \\
		            & = \pmu(y,w,w_0[2,c]).
	\end{align*}

	This concludes the proof that $\pJ' \neq \pJ$ and thus the description of the decomposition of $\Jtop$.
\end{proof}

\section{Subregular and submaximal cells in the exceptional types}

This section is purely computational. The left cells in types $G_2, F_4$ and $E_6$ were computed by Gibson and given in \cite{EJ}.

\begin{remark}
	We only describe the subregular cells in exceptional type. For \(G_2\) and
	\(F_4\), we also include the submaximal cells; the whole cell structure in
	these two types can be created with the code in \cite{MRT}, which also produced the
	cells displayed here. The code in \cite{MRT} also computes the \(p\)KL basis elements, but we do not list them here, since there are
		too many.
\end{remark}

\begin{theorem}[Exceptional subregular and submaximal \(p\)-cells]
	\label{prop:exceptional-subregular-submaximal-pcells}
	For all primes not listed below, the subregular cell remains a single
	two-sided \(p\)-cell. In types \(G_2\) and \(F_4\), the same convention applies
	to the submaximal cell when it is displayed below. The submaximal cells in
	types \(E_6,E_7,E_8\) are not included here.

	\begin{enumerate}[label=\textup{(\alph*)}]

		\item 	In type \(G_2\), the subregular and submaximal cells coincide. The only
		      exceptional primes are \(p=2,3\).
		      \begin{enumerate}[label=\textup{(\roman*)}]
			      \item For \(p=2\), this cell is the union of two two-sided \(p\)-cells, of
			            shapes
			            \[
				            \begin{array}{c|c}
					            2_{1,1} & 1_{1,1} \\ \hline
					            1_{1,1} & 2_{1,1}
				            \end{array}
				            \qquad\text{and}\qquad
				            1_{2,2}.
			            \]

			      \item For \(p=3\), this cell is the union of two two-sided \(p\)-cells, of
			            shapes
			            \[
				            \begin{array}{c|c}
					            2_{1,1} & 2_{1,1} \\ \hline
					            2_{1,1} & 3_{1,1}
				            \end{array}
				            \qquad\text{and}\qquad
				            1_{1,1}.
			            \]
		      \end{enumerate}
		      For \(p=3\) both cells contain a longest element of a parabolic subgroup,
		      while for \(p=2\) only the cell of shape
		      \(1_{2,2}\) does.

		\item In type \(F_4\), the only exceptional prime for the ordinary subregular and
		      submaximal cells is \(p=2\).
		      Here, both the ordinary subregular and submaximal cells are
		      the union of two two-sided \(p\)-cells, of shapes
		      \[
			      \begin{array}{c|c}
				      1_{2,2} & 1_{2,2} \\ \hline
				      1_{2,2} & 2_{2,2}
			      \end{array}
			      \qquad\text{and}\qquad
			      1_{2,2}.
		      \]
		      In the subregular case, both contain a longest element of a parabolic subgroup,
		      while in the submaximal case neither cell contains such an element.

		\item In types \(E_6,E_7,E_8\), the ordinary subregular cell remains a single
		      two-sided \(p\)-cell in every characteristic. Its shape is
		      \[
			      E_6:\ 1_{6,6},
			      \qquad
			      E_7:\ 1_{7,7},
			      \qquad
			      E_8:\ 1_{8,8}.
		      \]
		      All of these contain longest elements of a parabolic subgroup.
	\end{enumerate}
\end{theorem}

\begin{proof}
	Use the code in \cite{MRT}.
\end{proof}

\end{document}